\documentclass{amsart}

\usepackage{tikz}

\usetikzlibrary{decorations.markings}

\theoremstyle{plain}
  \newtheorem{thm}{Theorem}
  \newtheorem{defn}{Definition}
  \newtheorem{prop}{Proposition}
  \newtheorem{cor}[prop]{Corollary}
  \newtheorem{lem}[prop]{Lemma}
\theoremstyle{definition}
  \newtheorem{example}[prop]{Example}
  \newtheorem{rem}{Remark}

\renewcommand{\H}{\mathcal{H}}
\newcommand{\blf}{\mathbf{f}}
\newcommand{\ba}{\mathbf{a}}
\newcommand{\bb}{\mathbf{b}}
\newcommand{\bc}{\mathbf{c}}
\newcommand{\blg}{\mathbf{g}}
\newcommand{\im}{\mathrm{im}}
\newcommand{\dor}{\sigma}
\newcommand{\R}{\mathbb{R}}
\newcommand{\T}{\mathbb{T}}
\newcommand{\C}{\mathcal{C}}
\newcommand{\B}{\mathcal{B}}
\newcommand{\even}{\mathrm{even}}
\newcommand{\odd}{\mathrm{odd}}

\newcommand{\mm}[2]{m^{#1}_{#2}\Gamma\Delta}
\newcommand{\pp}[2]{P^{#1}_{#2}\Gamma\Delta}
\newcommand{\ppo}[2]{P^{#1}_{#2}\Gamma_0\Delta_0}
\newcommand{\ppt}[2]{P^{#1}_{#2}\Theta}
\newcommand{\CH}{\C_{\mathrm{Horn}}}

\newcommand{\ehor}{\bar ET}
\newcommand{\bint}{\B_{GZ}}
\newcommand{\gr}{\mathrm{gr}}
\begin{document}
\author{Anton Alekseev}

\address{Anton Alekseev, Department of Mathematics, University of Geneva, 2-4 rue du Li\`evre,
c.p. 64, 1211 Gen\`eve 4, Switzerland}

\email{Anton.Alekseev@unige.ch}

\author{Masha Podkopaeva}

\address{Maria Podkopaeva, Department of Mathematics, University of Geneva, 2-4 rue du Li\`evre,
c.p. 64, 1211 Gen\`eve 4, Switzerland}

\email{Maria.Podkopaeva@unige.ch}

\author{Andras Szenes}

\address{Andras Szenes, Department of Mathematics, University of Geneva, 2-4 rue du Li\`evre,
c.p. 64, 1211 Gen\`eve 4, Switzerland}

\email{Andras.Szenes@unige.ch}

\title[]{The Horn problem and planar networks}

\begin{abstract} The problem of determining the set of possible
  eigenvalues of 3 Hermitian matrices that sum up to zero is known as the
  Horn problem.  The answer is a polyhedral cone, which, following
  Knutson and Tao, can be described as the projection of a simpler cone
  in the space of  triangular tableaux (or {\em hives}) to the boundary nodes of the tableau.

  In this paper, we introduce a combinatorial problem defined in terms
  of certain weighted planar graphs giving rise to exactly
  the same polyhedral cone. In our framework, the values at the inner nodes
  of the triangular tableaux receive a natural interpretation. Other problems of linear algebra fit into the same scheme, among them the Gelfand--Zeitlin problem.
 Our approach is   motivated by the works of Fomin and Zelevinsky on total positivity and by the ideas of  tropicalization.

  \end{abstract}

\maketitle
\section{Introduction}

This article is motivated by the following classical problem of
linear algebra: under which conditions do three $n$-tuples of
ordered real numbers $\lambda_1\geq \dots\geq \lambda_n$, $\mu_1\geq
\dots\geq \mu_n$ and $\nu_1\geq \dots\geq \nu_n$ serve as the sets of
eigenvalues of three Hermitian $n$-by-$n$ matrices, $A$, $B$, and $C$, related by the equality $C=A+B$?

This problem has a long history (see \cite{fulton, brion} for
background). The first nontrivial necessary
condition,
\begin{equation*} \label{intro:e1}
\nu_1 \leq \lambda_1 + \mu_1,
\end{equation*}
was already known in the nineteenth century. Beginning with the work of
Weyl in 1912 \cite{weyl}, different sets of inequalities of
this type were found. Finally, in 1962, Horn put forward a very
complex, recursively defined set of conditions of the form
\[   \sum_{k\in K}  \nu_k\leq \sum_{i\in I} \lambda_I + \sum_{j\in J}
\mu_j,\text{ where }I,J,K\subset\{1,2,\dots, n\},
\]
which, together with the obvious equality  $ {\rm Tr}(A) + {\rm Tr}(B)
= {\rm Tr}(C)$, he conjectured to be necessary and sufficient.

Horn's conjecture was proved by Klyachko \cite{Klyachko}
and by Knutson and Tao \cite{KT}, \cite{puzzles}. Since then several other proofs
were given (e.g. \cite{KLM}), and the results were extended to
other problems of similar type (cf. \cite{AgnWood},
\cite{BelKum}, \cite{BerSja}).

In this paper, we will use the ``hive model'' of Horn's polyhedral
cone due to Knutson and Tao (see \cite{KT}, \cite{Buch}). Let $T_n$  be the $n$th order regular
triangulation of the equilateral triangle.  For real $n$-tuples
$\lambda_1 \geq \dots \geq \lambda_n$, $\mu_1 \geq \dots \geq \mu_n$,
and $\nu_1 \geq \dots \geq \nu_n$ satisfying
\begin{equation}\label{traces}
\sum_i\lambda_i+\sum_j\mu_j=\sum_k\nu_k,
\end{equation}
 we associate real numbers
to the boundary nodes of $T_n$ in the way demonstrated in Figure
\ref{intro:tri} for the case of $n=3$ (see Section \ref{sec:first} for
details).

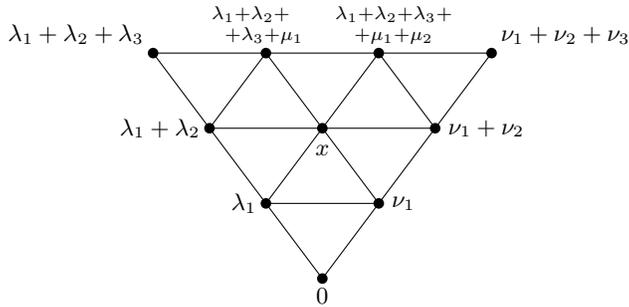
\begin{figure}[h]
\begin{tikzpicture}
\fill[black] (0,4) circle (2pt) node[above left,font=\small]{$\lambda_1+\lambda_2+\lambda_3$};
\fill[black] (1.5,4) circle (2pt);
\draw (1.3,4) node [above] {$\lambda_1+\lambda_2+\atop{\hskip 1em +\lambda_3+\mu_1}$};
\fill[black] (3,4) circle (2pt);
\draw (3.2,4) node[above]{$\lambda_1+\lambda_2+\lambda_3+\atop{+\mu_1+\mu_2}$};
\fill[black] (4.5,4) circle (2pt) node[above right,font=\small]{$\nu_1+\nu_2+\nu_3$};
\fill[black] (0.75,3) circle (2pt) node[left,font=\small]{$\lambda_1+\lambda_2$};
\fill[black] (2.25,3) circle (2pt);
\draw (2.25,2.9)  node[below,font=\small]{$x$};
\fill[black] (3.75,3) circle (2pt);
\draw (3.8,3) node[right,font=\small]{$\nu_1+\nu_2$};
\fill[black] (1.5,2) circle (2pt) node[left,font=\small]{$\lambda_1$};
\fill[black] (3,2) circle (2pt);
\draw (3.05,2) node[right,font=\small]{$\nu_1$};
\fill[black] (2.25,1) circle (2pt) node[below,font=\small]{$0$};
\draw (0,4) -- (4.5,4);
\draw (0.75,3)  -- (3.75,3);
\draw (1.5,2) -- (3,2);
\draw (0,4) -- (2.25,1) -- (4.5,4);
\draw (1.5,4) -- (3,2);
\draw (3,4) -- (3.75,3);
\draw (0.75,3) -- (1.5,4);
\draw (1.5,2) -- (3,4);
\end{tikzpicture}
\caption{Triangulation $T_3$.}\label{intro:tri}
\end{figure}

The Knutson--Tao theorem states that the three $n$-tuples may be
realized as ordered sets of eigenvalues of Hermitian matrices $A, B$, and
$C=A+B$ if and only if they satisfy the {\bf hive condition}:

\vskip 0.2cm

{\sl There exists a concave function $f$ defined on the
equilateral triangle,  linear on each small triangle of the
triangulation, and whose values at the boundary nodes of $T_n$
coincide with the values we ascribed to these nodes above. }

\vskip 0.2cm

Naturally, such a function $f$ is uniquely determined by its values on
the nodes of $T_n$. The condition of concavity translates into a set
of inequalities parametrized by the internal edges, or, equivalently,
by elementary rhombi of the triangulation. An example is the
inequality
\begin{equation} \label{intro:e2}
x \leq \lambda_1 + \nu_1,
\end{equation}
where $x$ is the value of $f$ at the central node in Figure
\ref{intro:tri}. 

In the present paper  motivated by constructions in the theory of
cluster algebras and total positivity (see \cite{FZ}), we introduce a
combinatorial framework  where the inequalities of the hive model
arise in a natural way. Instead of Hermitian matrices, we consider
certain oriented planar graphs, called {\em planar networks}, whose
edges are weighted by real (or tropical) numbers. An example of a
weighted planar network is shown in Figure \ref{intro:net}.

\begin{figure}[h]
\begin{tikzpicture}
\tikzset{->-/.style={decoration={markings, mark=at position .5 with {\arrow[scale=1.3]{>}}},postaction={decorate}}}
\draw[->-] (0,0) -- (2,0);
\draw[->-] (2,0) -- (6,0);
\draw[->-] (0,1) -- (1,1);
\draw[->-] (1,1) -- (3,1);
\draw[->-] (3,1) -- (6,1);
\draw[->-] (0,2) -- (5,2);
\draw[->-] (5,2) -- (6,2);
\draw[->-] (1,1) -- (2,0);
\draw[->-] (3,1) -- (5,2);
\draw (1,0.2) node[font=\small]{$0$};
\draw (4,0.2) node[font=\small]{$1$};
\draw (1.8,0.5) node[font=\small]{$1$};
\draw (0.5,1.2) node[font=\small]{$1$};
\draw (2,1.2) node[font=\small]{$1$};
\draw (4.5,1.2) node[font=\small]{$-1$};
\draw (3.8,1.6) node[font=\small]{$2$};
\draw (2.5,2.2) node[font=\small]{$1$};
\draw (5.5,2.2) node[font=\small]{$0$};
\draw[ultra thick,red] (0,1.02) -- (2.98,1.02) -- (4.98,2.02) -- (6,2.02);
\fill[black] (0,0) circle (1.5pt);
\fill[black] (0,1) circle (1.5pt);
\fill[black] (0,2) circle (1.5pt);
\fill[black] (6,0) circle (1.5pt);
\fill[black] (6,1) circle (1.5pt);
\fill[black] (6,2) circle (1.5pt);
\fill[black] (1,1) circle (1.5pt);
\fill[black] (2,0) circle (1.5pt);
\fill[black] (3,1) circle (1.5pt);
\fill[black] (5,2) circle (1.5pt);
\end{tikzpicture}
 \caption{A planar network. The maximal path is shown in thick red.}\label{intro:net}
\end{figure}
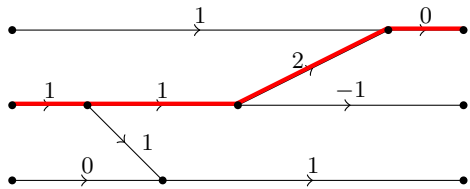

In this setup, the eigenvalues of a matrix correspond to certain
piecewise linear functions of the weights. The analogue of the top
eigenvalue $\lambda_1$, for example, will be the maximum of the
weights of oriented paths extending from the left to the right end of
the network (where the weight of a path is the sum of weights of the edges
contained in the path). For the weighted network in Figure
\ref{intro:net}, we obtain $\lambda_1= 1+1+2+0=4$.  The definition
of the other $\lambda$'s is given in Section \ref{sec:mi}; by analogy, we will
call these quantities the {\em eigenvalues} of the weighted network.

The addition of matrices is replaced by concatenation of planar
networks. In Figure \ref{intro:T12}, we give an example of a pair of
networks $\Gamma$ and $\Delta$ concatenated to form the network
$\Gamma\circ\Delta$. Providing $\Gamma$ and $\Delta$ with weights will
then allow us to fill in the boundary values of the tableau in Figure
\ref{intro:tri}: the value of $\nu_1$, for example, will be the
maximum of the weights of paths extending from left to right, while
$\lambda_1$ will be the maximum taken over paths extending from left
to the middle line $L$.

\begin{figure}[h]
\begin{tikzpicture}
\tikzset{->-/.style={decoration={markings, mark=at position .5 with {\arrow[scale=1]{>}}},postaction={decorate}}}
\draw[->-] (-0.2,0) -- (2,0);
\draw[->-] (2,0) -- (2.7,0);
\draw[->-] (2.7,0) -- (5.6,0);
\draw[->-] (1.5,1) -- (2,0);
\draw[->-] (-0.2,1) -- (1,1);
\draw[->-] (1,1) -- (1.5,1);
\draw[->-] (1.5,1) -- (2,1);
\draw[->-] (2,1) -- (2.7,1);
\draw[->-] (2.7,1) -- (5.6,1);
\draw[->-] (0.5,2) -- (1,1);
\draw[->-] (1.5,2) -- (2,1);
\draw[->-] (-0.2,2) -- (0.5,2);
\draw[->-] (0.5,2) -- (1.5,2);
\draw[->-] (1.5,2) -- (2.7,2);
\draw[->-] (2.7,2) -- (5.6,2);
\draw[ultra thick,red] (-0.2,1.02) -- (1.48,1.02) -- (1.98,0.02) -- (2.7,0.02);
\draw[ultra thick,red] (-0.2,2.02) -- (2.7,2.02);
\draw[ultra thick, red] (2.7,2.02) -- (5.6,2.02);
\fill[black] (-0.2,0) circle (1.5pt);
\fill[black] (-0.2,1) circle (1.5pt);
\fill[black] (-0.2,2) circle (1.5pt);
\fill[black] (2.7,0) circle (1.5pt);
\fill[black] (2.7,1) circle (1.5pt);
\fill[black] (2.7,2) circle (1.5pt);
\fill[black] (5.6,0) circle (1.5pt);
\fill[black] (5.6,1) circle (1.5pt);
\fill[black] (5.6,2) circle (1.5pt);
\fill[black] (2,0) circle (1.5pt);
\fill[black] (1,1) circle (1.5pt);
\fill[black] (1.5,1) circle (1.5pt);
\fill[black] (2,1) circle (1.5pt);
\fill[black] (0.5,2) circle (1.5pt);
\fill[black] (1.5,2) circle (1.5pt);
\draw[opacity=0.3] (2.7,-0.3)--(2.7,2.3) node[above] {$L$};
\draw(1.3,2.5) node{$\Gamma$};
\draw(4.2,2.5) node{$\Delta$};
\end{tikzpicture}
 \caption{Concatenation of networks $\Gamma$ and $\Delta$. A path is shown in thick red.}\label{intro:T12}
\end{figure}
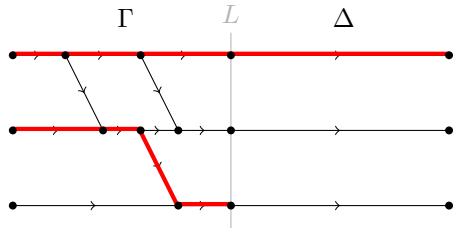

Now we can formulate our first correspondence result, Theorem
\ref{thm:Horn1}: for any pair of weighted planar networks $\Gamma$ and
$\Delta$, the eigenvalues of $\Gamma$, $\Delta$, and
$\Gamma\circ\Delta$ satisfy the hive condition and the trace condition \eqref{traces}.

A key feature of our construction is that there is a natural  explicit
definition of the values of the function $f$ at the internal nodes of
$T_n$, and this makes the proofs rather straightforward.  For
example, the value $x$ of
$f$ assigned to the middle node in Figure \ref{intro:tri}, is the maximal sum of the weights of two {\bf disjoint} oriented
paths in $\Gamma\circ\Delta$, one of which extends from left to right,
while the other one from left to the middle line $L$.  An example of such
a pair is shown in Figure \ref{intro:T12}.  Note that  inequality
\eqref{intro:e2} is now obvious: on the right-hand side of (2), the maximum is taken over all pairs of paths, one going from left to right and the other one from left to the middle, while on the left-hand side of (2), the maximum is taken only over a subset of such pairs, namely, over all pair of disjoint pathes.

A natural question is to compare the set defined by these eigenvalue analogs with the set of eigenvalues of triples of Hermitian matrices.  Theorem
\ref{thm:Horn2} states that any triple of ordered $n$-tuples satisfying
the hive condition is the triple of the sets of eigenvalues of $\Gamma$,
$\Delta$, and $\Gamma\circ\Delta$ for some  weighted
networks $\Gamma$ and $\Delta$. For $n=3$, the
graphs underlying $\Gamma$ and $\Delta$ can be chosen as shown in Figure \ref{intro:T12}.

We begin our paper with a planar network interpretation of the
precursor of the Horn problem: the interlacing inequalities for
eigenvalues of a Hermitian matrix and its principal submatrices (see
\cite{HJ}).  These latter inequalities play a prominent role in the description of  the
Gelfand--Zeitlin integrable system (see \cite{GuillStern}).  Then we
proceed to prove our main results, Theorems \ref{thm:Horn1} and
\ref{thm:Horn2}.

One can regard the eigenvalue problem for planar networks as a rather
nontrivial tropicalization of the eigenvalue problem for Hermitian
matrices. From this perspective, it is natural
to expect  but not very easy to prove that both problems are governed
by the same set of inequalities.  In the forthcoming paper \cite{APS},
we will prove this ``detropicalization'' correspondence principle for
the Horn problem, thus providing a new proof of the theorem of Knutson
and Tao.

\vskip 0.2cm

{\bf Acknowledgements.} We are grateful to S. Fomin, A. Knutson, E. Meinrenken,  M. Vergne, J. Weitsmann, C. Woodward, and A. Zelevinsky for inspiring discussions and comments.

\section{Gelfand--Zeitlin and Horn problems}\label{sec:first}

 In this section, we recall two classical problems of linear
 algebra: the Gelfand--Zeitlin and Horn problems.

\subsection{The Gelfand--Zeitlin problem}

Let $\H_n$ be the set of $n$-by-$n$ Hermitian matrices. For $A \in
\H_n$ and $1 \leq k \leq n$, denote by $A^{(k)}$ the principal
submatrix of $A$ of size $k$, i.e., the $k$-by-$k$ submatrix sitting
in the upper left corner of $A$.  Let $(\lambda^{(k)}_1 \geq \dots
\geq \lambda^{(k)}_k)$ be the sequence of ordered eigenvalues of
$A^{(k)}$; this way we obtain a set of functions
\begin{equation}
  \label{minorev}
  \lambda_i^{(k)}: \H_n \to \R, \quad 0< i \leq k \leq n.
\end{equation}

It is a classical  result of linear algebra (see, e.g., \cite{HJ}) that these
functions satisfy the following
{\em interlacing inequalities}:
\begin{equation}
  \label{eq:GZ}
\lambda^{(k+1)}_i \geq \lambda^{(k)}_i \geq \lambda^{(k+1)}_{i+1},\;
0 < i \leq k < n.
\end{equation}
The converse is also true: any set of numbers satisfying the interlacing inequalities appears as the eigenvalues of a Hermitian matrix and its principal submatrices.

One can recast inequalities \eqref{eq:GZ} in the following form.
Let $T=T_n$ be a regular triangulation of the equilateral triangle
(Figure \ref{tableau}) with the set of nodes $VT$ parametrized by the indices
$0\leq i\leq k\leq n$; thus we have $|VT|=(n+1)(n+2)/2$.  To each node of the triangulation, we associate
a coordinate variable
$t^k_i:\R^{VT}\to\R$.

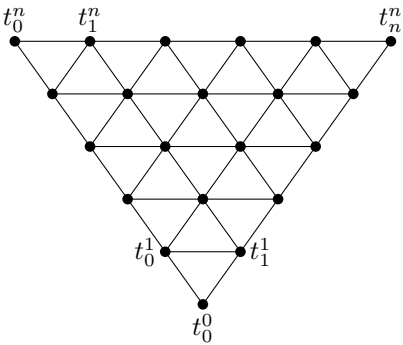
\begin{figure}[h]
\begin{tikzpicture}
\fill[black] (0,6) circle (2pt) node[above]{$t_0^n$};
\fill[black] (1,6) circle (2pt) node[above]{$t_1^n$};
\fill[black] (2,6) circle (2pt);
\fill[black] (3,6) circle (2pt);
\fill[black] (4,6) circle (2pt);
\fill[black] (5,6) circle (2pt) node[above]{$t_n^n$};
\fill[black] (0.5,5.3) circle (2pt);
\fill[black] (1.5,5.3) circle (2pt);
\fill[black] (2.5,5.3) circle (2pt);
\fill[black] (3.5,5.3) circle (2pt);
\fill[black] (4.5,5.3) circle (2pt);
\fill[black] (1,4.6) circle (2pt);
\fill[black] (2,4.6) circle (2pt);
\fill[black] (3,4.6) circle (2pt);
\fill[black] (4,4.6) circle (2pt);
\fill[black] (1.5,3.9) circle (2pt);
\fill[black] (2.5,3.9) circle (2pt);
\fill[black] (3.5,3.9) circle (2pt);
\fill[black] (2,3.2) circle (2pt) node[left]{$t_0^1$};
\fill[black] (3,3.2) circle (2pt) node[right]{$t_1^1$};
\fill[black] (2.5,2.5) circle (2pt) node[below]{$t_0^0$};
\draw (0,6) -- (5,6);
\draw (0.5,5.3)  -- (4.5,5.3);
\draw (1,4.6) -- (4,4.6);
\draw (1.5,3.9) -- (3.5,3.9);
\draw (2,3.2) -- (3,3.2);
\draw (0,6) -- (2.5,2.5) -- (5,6);
\draw (1,6) -- (3,3.2);
\draw (2,6) -- (3.5,3.9);
\draw (3,6) -- (4,4.6);
\draw (4,6) -- (4.5,5.3);
\draw (4,6) -- (2,3.2);
\draw (3,6) -- (1.5,3.9);
\draw (2,6) -- (1,4.6);
\draw (1,6) -- (0.5,5.3);
\end{tikzpicture}
\caption{Triangulation $T_n$.}\label{tableau}
\end{figure}

Denote by $\ehor$ the set of horizontal edges of the triangulation $T$ 
parameterized by pairs $(i,k)$ satisfying $0< i \leq k \leq n$; again,
to each horizontal edge, we associate coordinate functions on
$\R^{\ehor}$, $h_i^{(k)}:\R^{\ehor}\to\R$, $0< i \leq k \leq
n$. Note that this index set coincides with the index set of  the eigenvalues
of the principal submatrices of an Hermitian matrix \eqref{minorev}.
\begin{defn}  \label{defc2}
The cone $\C_2\subset\R^{VT}$ is the polyhedral cone defined by the system of
 inequalities
\begin{equation}  \label{eq:interlacing}
\begin{array}{lll}
t^{k+1}_{i} + t^{k}_{i-1} & \geq & t^{k+1}_{i-1} + t^{k}_{i},  \\
t^{k+1}_{i} + t^{k}_{i} & \geq & t^{k+1}_{i+1} + t^{k}_{i-1}
\end{array}
\end{equation}
for $0 < i \leq k < n$. 

The horizontal boundary map is the map 
\[    \bar\partial:\R^{VT}\to\R^{\ehor}:\; \{t_i^k\}_{0< i \leq k \leq n}\mapsto \{h_i^{(k)}=t^k_i-t^k_{i-1}\}_{0< i \leq k \leq n}.
\]
\end{defn}

Note that   inequalities \eqref{eq:interlacing} are parametrized by
the internal non-horizontal edges of the triangulation $T$, or,
alternatively, by rhombi of two types having these edges as   short
diagonals (see Figure \ref{rhombi12}): for each rhombus, the corresponding
inequality states that the sum of the two numbers assigned to the
endpoints of the short diagonal is greater than or equal to the sum of the
two numbers assigned to the endpoints of the long diagonal.

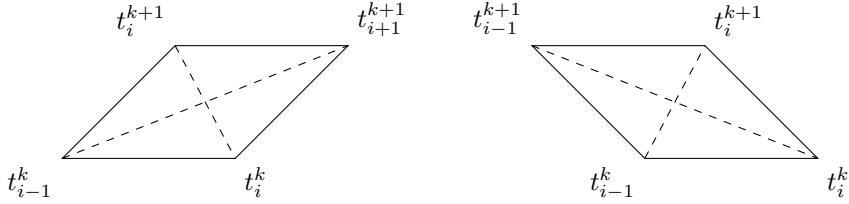
\begin{figure}[h]
\begin{minipage}{0.48\linewidth}
\begin{tikzpicture}
\draw (0,0) node[below left]{$t_{i-1}^k$} -- (1.5,1.5) node[above left]{$t_{i}^{k+1}$} -- (3.8,1.5) node[above right]{$t_{i+1}^{k+1}$} -- (2.3,0) node[below right]{$t_{i}^k$} -- cycle;
\draw [dashed] (0,0) -- (3.8,1.5);
\draw [dashed] (1.5,1.5) -- (2.3,0);
\end{tikzpicture}
\end{minipage}
\begin{minipage}{0.48\linewidth}
\begin{tikzpicture}
\draw (1.5,0) node[below left]{$t_{i-1}^k$} -- (0,1.5) node[above left]{$t_{i-1}^{k+1}$} -- (2.3,1.5)  node[above right]{$t_{i}^{k+1}$} -- (3.8,0) node[below right]{$t_{i}^k$} -- cycle;
\draw [dashed] (1.5,0) -- (2.3,1.5);
\draw [dashed] (0,1.5) -- (3.8,0);
\end{tikzpicture}
\end{minipage}
\caption{The two types of rhombi.}\label{rhombi12}
\end{figure}

Observe that
\begin{itemize}
\item the set $\C_2$ and the linear map $\bar\partial$ are invariant under the
  transformations $t^k_i\mapsto t^k_i+c_k$ for
  $(c_k,\dots,c_0)\in\R^{k+1}$,  
\item the linear map $\bar\partial$ establishes a linear isomorphism
  between  $\{t_0^k=0|\;k=0,\dots,n\}\subset\R^{VT}$ and $\R^{\ehor}$.
\end{itemize}

Then we have the following characterization of the interlacing inequalities.
\begin{prop}\label{thm:GZ}
  Let $\C_{GZ}\subset\R^{\ehor}$ be the cone defined by the interlacing inequalities \eqref{eq:GZ}.
Then
\[   \C_{GZ} =\bar\partial(\C_2).
\]
In fact, the linear map $\bar\partial$ establishes an isomorphism of
  polyhedral cones:
\begin{equation}
  \label{imc2}
   \bar\partial:\C_2\cap\{t_0^k=0|\;k=0,\dots,n\} \to \C_{GZ}.
\end{equation}
\end{prop}
\begin{proof}
Let $t=\{t_i^k\}_{i,k}$ be in $\C_2$. Then we can rewrite inequalities \eqref{eq:interlacing} as
\[
\begin{array}{lll}
t^{k+1}_{i} -  t^{k+1}_{i-1} & \geq & t^{k}_{i} - t^{k}_{i-1},  \\
t^{k}_{i} - t^{k}_{i-1} & \geq & t^{k+1}_{i+1} - t^{k+1}_{i}.
\end{array}
\]
Therefore, the image $h=\{h_i^k\}_{i,k}$ of $t$ under the map $\bar\partial$ satisfies the condition
\[h^{k+1}_i\geq h^k_i\geq h^{k+1}_{i+1},\]
i.e., the interlacing inequalities.

Conversely, given a point $h=\{h_i^k\}_{i,k}$ in $\C_{GZ}$, define $t_i^k=h_1^k+\dots+h_i^k$  for all $i$ and $k$. Then $h=\bar\partial t$, and the interlacing inequalities for $h_i^k$ imply precisely the inequalities \eqref{eq:interlacing} for $t_i^k$.
\end{proof}

\subsection{The Horn problem}


The Horn problem is related to the eigenvalues of triples of Hermitian
matrices that add up to zero.  More formally, consider the eigenvalue
map
\[  \Lambda^{\times3}: \H_n\times \H_n\times
\H_n\to  \R^{3n} 
\]
listing the eigenvalues of a triple of Hermitian
matrices of rank $n$, where the eigenvalues of each matrix are listed in decreasing order.  The {\em Horn cone} is defined as the image
\[ \CH = \Lambda^{\times3}
\left(\{(A,B,C)\in\H_n^{\times3};\,A+B=C\}\right)\subset\R^{3n}.
\]
Note that $A+B=C$ implies ${\rm Tr}(A)+{\rm Tr}(B) = {\rm Tr}(C)$
since the trace is a linear functional. For the Horn cone, this means
that if $$(\lambda_1,\dots,\lambda_n,\mu_1,\dots,\mu_n,\nu_1,\dots,\nu_n)\in\CH,$$
then
\begin{equation} \label{eq:Tr=0}
\sum_{i=1}^n \lambda_i + \sum_{i=1}^n \mu_i = \sum_{i=1}^n \nu_i.
\end{equation}

In \cite{Horn}, Horn gave a rather complicated set of inequalities
 that would conjecturally define $\CH$; in particular, he suggested
that $\CH$ is a polyhedral cone inside the hyperplane given by
Equation \eqref{eq:Tr=0}. His conjecture was proved more than 30 years
later \cite{Klyachko, KT}. In the present paper, we will use the following
elegant description of $\CH$  due to Knutson and Tao \cite{KT}.

Consider again the triangulation $T=T_n$ and define the subcone
$\C_3\subset\C_2$ cut out from $\C_2$ by the inequalities corresponding to
the third set of rhombi (see Figure \ref{rhombi3}).

\begin{figure}[h]
\begin{tikzpicture}
\draw (0,0) node[below]{$t_{i-1}^{k-1}$} -- (-0.8,1.1) node[ left]{$t_{i-1}^{k}$} -- (0,2.2) node[above]{$t_{i}^{k+1}$} -- (0.8,1.1) node[right]{$t_{i}^k$} -- cycle;
\draw [dashed] (0,0) -- (0,2.2);
\draw [dashed] (-0.8,1.1) -- (0.8,1.1);
\end{tikzpicture}
\caption{Third type of rhombi.}\label{rhombi3}
\end{figure}
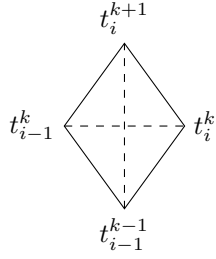

More precisely, we
define the polyhedral cone $\C_3\subset\R^{VT}$ by the following set
of inequalities:
\begin{equation}  \label{eq:KT}
\begin{array}{lll}
t^{k+1}_{i} + t^{k}_{i-1} & \geq & t^{k+1}_{i-1} + t^{k}_{i},  \\
t^{k+1}_{i} + t^{k}_{i} & \geq & t^{k+1}_{i+1} + t^{k}_{i-1},\\
t^{k}_{i} + t^{k}_{i-1} & \geq & t^{k+1}_{i} + t^{k-1}_{i-1}
\end{array}
\end{equation}
for $0 < i \leq k < n$.

\begin{rem} \label{trans3}
We note that $\C_3$ is invariant under the translations $t^k_i\to
  t^k_i+c$  for $c\in\R$.
\end{rem}

Now we define the map $\partial: \R^{VT} \to \R^{3n}$ by the formula
\[ \partial:\ \{t_i^k,\ 0\leq i\leq k\leq n\}\mapsto \{\lambda_i=t^i_0-t^{i-1}_0,\ \mu_i=t^n_i-t^n_{i-1},\
       \nu_i = t^i_i-t^{i-1}_{i-1},\ 1\leq i\leq n\}.
\]
In fact,  $\partial$ is simply the cohomological boundary operator restricted to the outer edges of the triangulation $T_n$.
Then Horn's conjecture may be formulated as follows.
\begin{thm}[Knutson--Tao] \label{thm:KT}
 $\partial(\C_3)=\CH$.
\end{thm}
\begin{rem}
  Since the operator $\partial$ is invariant under the translation
  mentioned in Remark \ref{trans3}, we can normalize $t^0_0=0$, and then the
  theorem may be equivalently stated as follows:
  \begin{equation}
    \label{thm:horn0}
    \partial(\C_3\cap\{t^0_0=0\})=\CH.
  \end{equation}
\end{rem}

\section{Planar networks}\label{sec:mi}

In this section, we introduce the notion of planar networks, the key tool for the rest of the paper.
There are a number of possible definitions of planar networks. We chose the one
that was  the most convenient for our purposes, but the proofs could
be adapted to the other definitions as well.

\begin{defn}
  A  planar network is the following data:
  \begin{itemize}
  \item a finite  graph $\Gamma$ with vertex set $V\Gamma$ and edge
    set $E\Gamma$,
  \item a pair of reals, $a<b$,
  \item an embedding of $\Gamma$ into the strip $\{a\leq x \leq
    b\}\subset\R^2$ such that the image of each edge is a segment of a
    straight line, which is not parallel to the $y$-axis.
  \end{itemize}
  \end{defn}
We will call the vertices on the line $\{x=a\}$
  {\em sources} and the vertices on the line $\{x=b\}$
 {\em sinks} of $\Gamma$; the other vertices will be called {\em internal}.

Observations:
\begin{itemize}
\item A {\em subnetwork} of a planar network $\Gamma$ is naturally
  defined as a subgraph of $\Gamma$, with the rest of the data unchanged. The sources and sinks of the
  subnetwork thus have to be subsets of the sources and sinks of
  $\Gamma$.
\item A planar network $\Gamma$ is naturally oriented (from left to
  right), and we will use this orientation in what follows.
\item Using this orientation, one can characterize the vertices of
  $\Gamma$ by the number of incoming and outgoing edges. A source,
  for example, is always a vertex of degree $(0,d)$ for some nonnegative integer
  $d$.
\end{itemize}

A crucial role in our analysis will be played by paths.
\begin{defn} A  \em multipath  in $\Gamma$ is a subnetwork  whose
  every vertex that is an internal vertex of $\Gamma$ is of degree
  $(1,1)$.
\end{defn}
A multipath has the same number of sources and sinks; a multipath with k sources is called a {\em $k$-path}. Each $k$-path is simply the union of $k$
disjoint paths of $\Gamma$ connecting a source with a sink. The set of
$k$-paths in $\Gamma$ will be denoted by $P_k\Gamma$.

\begin{defn} Let $\T=\R\cup\{-\infty\}$ be the semifield of tropical numbers.
  A {\em weighting} of a planar network $\Gamma$ is an assignment
  of a tropical number to each  edge of $\Gamma$.
Identifying the set of weightings of $\Gamma$ with  $\T^{E\Gamma}$, we
can introduce
\begin{itemize}
\item  the coordinate function $w_e:\T^{E\Gamma}\to \T$  for each edge
  $e\in E\Gamma$, and
\item the weight functional
\[   w_\alpha:\T^{E\Gamma}\to\T,\quad w_\alpha = \sum_{e\in E\alpha}w_e 
\]
for each subgraph $\alpha$ of $\Gamma$, in particular, for each
  multipath in $\Gamma$. When $\alpha$ has no edges, we set $w_{\alpha}=0$.
\end{itemize}
\end{defn}

\subsection{The maximum functionals} A certain collection of
piecewise linear functions on the space $\T^{E\Gamma}$ of weightings of
a planar network $\Gamma$ will play an important role in what follows.
For a planar network $\Gamma$, $i>0$,  and a weighting $\epsilon\in\T^{E\Gamma}$, we define
$l_i\Gamma:\T^{E\Gamma}\to\T$ by
\begin{equation} \label{eq:lGamma}
l_i\Gamma(\epsilon) = {\rm max}\{ w_\alpha(\epsilon)|\;\alpha \in P_i\Gamma\} 
\end{equation}
if the set $P_i\Gamma$ is nonempty; otherwise, we set
$l_i\Gamma=-\infty$.
By definition, we put $l_0\Gamma=0$ and
 denote by $l\Gamma$ the $(n+1)$-tuple $(l_0\Gamma,\dots,\l_n\Gamma)$.
\begin{rem}
If, for a weighting $\epsilon$, we have $l_i\Gamma(\epsilon)\in\R$  for all $i$, then we can define the eigenvalues associated to $\epsilon$ by the formula
$$\lambda_i=l_i\Gamma(\epsilon)-l_{i-1}\Gamma(\epsilon),$$
so as $l_i\Gamma(\epsilon)=\lambda_1+\dots+\lambda_i$.
\end{rem}

\begin{example}
The simplest example is a planar network $\Gamma$ that contains
no edges (Figure \ref{net:empty}). Then $\T^{E(\Gamma)}$ is a single point, and the image of
$l\Gamma$ is the point $(0,-\infty, \dots, -\infty)$.
\end{example}

\begin{example}
The next example is a planar network with exactly $n$
edges $e_i$ connecting the vertices $(a,i)$ and $(b,i)$  (Figure \ref{net:straight}). Denote the
corresponding weights by $w_i$ and let $(\varpi_1, \varpi_2, \dots, \varpi_n)$
be the permutation of the $n$-tuple $(w_1, w_2, \dots, w_n)$
such that $\varpi_i \geq \varpi_{i+1}$ for all $i=1,  \dots, n-1$.
Then
$$
l_i\Gamma= \varpi_1 + \dots + \varpi_i.
$$
The image of $l\Gamma$ is the closure in $\T^n$ of the polyhedral
cone defined by the inequalities $\varpi_i \geq \varpi_{i+1}$.
\end{example}

\begin{figure}[h]
\begin{minipage}{0.48\linewidth}
\begin{tikzpicture}

\fill[black] (0,0) circle (1.5pt);
\fill[black] (0,0.4) circle (1.5pt);
\fill[black] (0,1) circle (1.5pt);
\fill[black] (0,1.3) circle (1.5pt);
\fill[black] (0,2) circle (1.5pt);
\fill[black] (4,0) circle (1.5pt);
\fill[black] (4,0.6) circle (1.5pt);
\fill[black] (4,1.2) circle (1.5pt);
\fill[black] (4,1.5) circle (1.5pt);
\fill[black] (4,1.9) circle (1.5pt);
\fill[black] (4,2.2) circle (1.5pt);
\draw[->] (-0.6,-0.5) -- (4.5,-0.5) node[right]{$x$};
\draw[->] (-0.6,-0.5) -- (-0.6,2.6) node[above]{$y$};
\path (0,-0.8) node {$a$};
\path (4,-0.8) node {$b$};
\draw[loosely dotted,gray] (0,-0.5) -- (0,2.4);
\draw[loosely dotted,gray] (4,-0.5) -- (4,2.4);
\end{tikzpicture}
 \caption{}\label{net:empty}\end{minipage}
\begin{minipage}{0.48\linewidth}
\begin{tikzpicture}
\draw (0,0) -- (4,0);
\draw (0,0.5) -- (4,0.5);
\draw (0,1) -- (4,1);
\draw (0,1.5) -- (4,1.5);
\draw (0,2) -- (4,2);
\fill[black] (0,0) circle (1.5pt) node[left]{$1$};
\fill[black] (0,0.5) circle (1.5pt) node[left]{$2$};
\fill[black] (0,1) circle (1.5pt);
\fill[black] (0,1.5) circle (1.5pt);
\fill[black] (0,2) circle (1.5pt) node[left]{$n$};
\fill[black] (4,0) circle (1.5pt) node[right]{$1$};
\fill[black] (4,0.5) circle (1.5pt) node[right]{$2$};
\fill[black] (4,1) circle (1.5pt);
\fill[black] (4,1.5) circle (1.5pt);
\fill[black] (4,2) circle (1.5pt) node[right]{$n$};
\draw[->] (-0.6,-0.5) -- (4.5,-0.5) node[right]{$x$};
\draw[->] (-0.6,-0.5) -- (-0.6,2.6) node[above]{$y$};
\path (0,-0.8) node {$a$};
\path (4,-0.8) node {$b$};
\draw[loosely dotted,gray] (0,-0.5) -- (0,2.4);
\draw[loosely dotted,gray] (4,-0.5) -- (4,2.4);
\end{tikzpicture}
 \caption{}\label{net:straight}
\end{minipage}
\end{figure}

\begin{lem}\label{subnet}
  \begin{enumerate}
  \item   If $\Gamma'$ is a subnetwork of $\Gamma$, then $ \im(l\Gamma')
  \subset \im(l\Gamma).  $
\item  If the network $\Gamma'$ can be obtained from $\Gamma$ by insertion of a new vertex splitting an existing edge, then  $ \im(l\Gamma')
  = \im(l\Gamma).  $
  \end{enumerate}
\end{lem}

\begin{proof}
  1.  Set the weights of all edges $e \in E\Gamma\backslash E\Gamma'$
  equal to $-\infty$. The image of this subset of weightings under $l\Gamma$
  coincides with the image of $l\Gamma'$.\\
  2. Let 
  $s:\T^{E\Gamma'}\to\T^{E\Gamma}$ be the map assigning to the split edge the
  sum of the weights of the two edges obtained by the insertion of
  the new vertex. Then $s$ preserves the functional $l$:
  $l\Gamma(s\epsilon)=l\Gamma'(\epsilon)$. This implies $
  \im(l\Gamma') = \im(l\Gamma)$.
\end{proof}
Our main goal is to study the image of the piecewise linear map
$l\Gamma$  for a planar network $\Gamma$. We will need the following consequence of
Lemma \ref{subnet}.
\begin{cor}  \label{lemcor}
By allowing embedded edges that are unions of intervals,
we can always replace a planar network $\Gamma$ by another planar
network $\Gamma'$ having no vertices of
degree (1,1) so that $\im(l\Gamma)=\im(l\Gamma')$.
\end{cor}

\section{Main results}

In this section, we state the main results of the paper.
We establish a correspondence principle between the functionals $l\Gamma$ for
planar networks $\Gamma$ and eigenvalues of Hermitian matrices.
In particular, we show that the Gelfand--Zeitlin cone and the Horn cone
appear as images of natural piecewise linear functions on the space
of weightings of planar networks.

\subsection{Planar networks and Gelfand--Zeitlin}
In this section, we will assume that the planar network $\Gamma$ has
precisely $n$ sources and sinks. Without loss of generality, we can
assume that the set of $y$-coordinates of the sources and sinks is the
set of the first $n$ integers $\{1,2,\dots,n\}$. We will say that such a network
$\Gamma$ is a planar network {\em of rank $n$}.

For a planar network of rank $n$, we denote by $\Gamma^{(k)}$ the
maximal subgraph of $\Gamma$ that does not contain the sinks or
sources with $y$ coordinates above the line $\{y=k\}$; these are the vertices
\[  (a,k+1),(b,k+1),\dots,(a,n),(b,n).
\]
Then $\Gamma^{(k)}$ is a planar network of rank $k$.

The collection of maps $l_i{\Gamma^{(k)}}$, $0\leq i\leq k\leq n$,
defines a map from the set of weightings of $\Gamma$ to the set of
triangular tableaux (Figure \ref{tableau}) filled by tropical numbers:
\[    L\Gamma: \T^{E\Gamma}\to \T^{VT}.
\]
 Note that each row of the tableau gives the map $l\Gamma^{(k)}$ for the appropriate $k$.

\begin{thm} \label{thm:GZ1}
Let $\Gamma$ be a planar network of rank $n$. Then
$$
{\rm im}(L\Gamma) \subset \overline{\C_2}\cap\{t_0^k=0|\;k=0,\dots,n\}.
$$
\end{thm}
Here and below the closure is  taken in $\T^N$.

The image of $L\Gamma$ depends on the planar network
$\Gamma$. There are networks, however, for which this image is
maximal. Let $\Gamma_0=\Gamma_0[n]$ be the network in Figure
\ref{net:triang}.

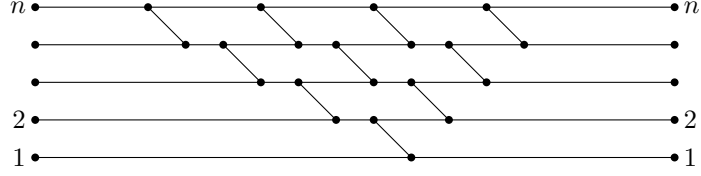
\begin{figure}[h]
\begin{tikzpicture}
\draw (-1,0) -- (7.5,0);
\draw (-1,0.5) -- (7.5,0.5);
\draw (-1,1) -- (7.5,1);
\draw (-1,1.5) -- (7.5,1.5);
\draw (-1,2) -- (7.5,2);
\draw (0.5,2) -- (1,1.5);
\draw (2,2) -- (2.5,1.5);
\draw (3.5,2) -- (4,1.5);
\draw (5,2) -- (5.5,1.5);
\draw (1.5,1.5) -- (2,1);
\draw (3,1.5) -- (3.5,1);
\draw (4.5,1.5) -- (5,1);
\draw (2.5,1) -- (3,0.5);
\draw (4,1) -- (4.5,0.5);
\draw (3.5,0.5) -- (4,0);
\fill[black] (-1,0) circle (1.5pt) node[left]{$1$};
\fill[black] (-1,0.5) circle (1.5pt) node[left]{$2$};
\fill[black] (-1,1) circle (1.5pt);
\fill[black] (-1,1.5) circle (1.5pt);
\fill[black] (-1,2) circle (1.5pt) node[left]{$n$};
\fill[black] (7.5,0) circle (1.5pt) node[right]{$1$};
\fill[black] (7.5,0.5) circle (1.5pt) node[right]{$2$};
\fill[black] (7.5,1) circle (1.5pt);
\fill[black] (7.5,1.5) circle (1.5pt);
\fill[black] (7.5,2) circle (1.5pt) node[right]{$n$};
\fill[black] (0.5,2) circle (1.5pt);
\fill[black] (1,1.5) circle (1.5pt);
\fill[black] (2,2) circle (1.5pt);
\fill[black] (2.5,1.5) circle (1.5pt);
\fill[black] (3.5,2) circle (1.5pt);
\fill[black] (4,1.5) circle (1.5pt);
\fill[black] (5,2) circle (1.5pt);
\fill[black] (5.5,1.5) circle (1.5pt);
\fill[black] (1.5,1.5) circle (1.5pt);
\fill[black] (2,1) circle (1.5pt);
\fill[black] (3,1.5) circle (1.5pt);
\fill[black] (3.5,1) circle (1.5pt);
\fill[black] (4.5,1.5) circle (1.5pt);
\fill[black] (5,1) circle (1.5pt);
\fill[black] (2.5,1) circle (1.5pt);
\fill[black] (3,0.5) circle (1.5pt);
\fill[black] (4,0) circle (1.5pt);
\fill[black] (4,1) circle (1.5pt);
\fill[black] (4.5,0.5) circle (1.5pt);
\fill[black] (3.5,0.5) circle (1.5pt);
\end{tikzpicture}
\caption{The planar network $\Gamma_0$.}\label{net:triang}
\end{figure}

\begin{rem}
Clearly, 
eliminating from a rank-$n$ planar network $\Gamma$ all vertices (with
the adjacent edges) that cannot be reached from a source, we do not change the
image of the functional $L\Gamma$.
Combining this with Corollary \ref{lemcor}, we see that we
can replace the subnetwork $\Gamma_0^{(k)}\subset\Gamma_0$ with the network
$\Gamma_0[k]$ (see Figure \ref{net:sub}).
\end{rem}

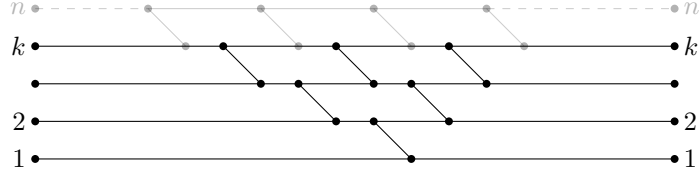
\begin{figure}[h]
\begin{tikzpicture}
\draw (-1,0) -- (7.5,0);
\draw (-1,0.5) -- (7.5,0.5);
\draw (-1,1) -- (7.5,1);
\draw (-1,1.5) -- (7.5,1.5);
\draw[opacity=0.2] (0.5,2) -- (5,2);
\draw[opacity=0.2,dashed] (-1,2) -- (0.5,2);
\draw[opacity=0.2,dashed] (5,2) -- (7.5,2);
\draw[opacity=0.2] (0.5,2) -- (1,1.5);
\draw[opacity=0.2] (2,2) -- (2.5,1.5);
\draw[opacity=0.2] (3.5,2) -- (4,1.5);
\draw[opacity=0.2] (5,2) -- (5.5,1.5);
\draw (1.5,1.5) -- (2,1);
\draw (3,1.5) -- (3.5,1);
\draw (4.5,1.5) -- (5,1);
\draw (2.5,1) -- (3,0.5);
\draw (4,1) -- (4.5,0.5);
\draw (3.5,0.5) -- (4,0);
\fill[black] (-1,0) circle (1.5pt) node[left]{$1$};
\fill[black] (-1,0.5) circle (1.5pt) node[left]{$2$};
\fill[black] (-1,1) circle (1.5pt);
\fill[black] (-1,1.5) circle (1.5pt) node[left]{$k$};
\fill[black,opacity=0.3] (-1,2) circle (1.5pt) node[left]{$n$};
\fill[black] (7.5,0) circle (1.5pt) node[right]{$1$};
\fill[black] (7.5,0.5) circle (1.5pt) node[right]{$2$};
\fill[black] (7.5,1) circle (1.5pt);
\fill[black] (7.5,1.5) circle (1.5pt) node[right]{$k$};
\fill[black,opacity=0.3] (7.5,2) circle (1.5pt) node[right]{$n$};
\fill[black,opacity=0.3] (0.5,2) circle (1.5pt);
\fill[black,opacity=0.3] (1,1.5) circle (1.5pt);
\fill[black,opacity=0.3] (2,2) circle (1.5pt);
\fill[black,opacity=0.3] (2.5,1.5) circle (1.5pt);
\fill[black,opacity=0.3] (3.5,2) circle (1.5pt);
\fill[black,opacity=0.3] (4,1.5) circle (1.5pt);
\fill[black,opacity=0.3] (5,2) circle (1.5pt);
\fill[black,opacity=0.3] (5.5,1.5) circle (1.5pt);
\fill[black] (1.5,1.5) circle (1.5pt);
\fill[black] (2,1) circle (1.5pt);
\fill[black] (3,1.5) circle (1.5pt);
\fill[black] (3.5,1) circle (1.5pt);
\fill[black] (4.5,1.5) circle (1.5pt);
\fill[black] (5,1) circle (1.5pt);
\fill[black] (2.5,1) circle (1.5pt);
\fill[black] (3,0.5) circle (1.5pt);
\fill[black] (4,0) circle (1.5pt);
\fill[black] (4,1) circle (1.5pt);
\fill[black] (4.5,0.5) circle (1.5pt);
\fill[black] (3.5,0.5) circle (1.5pt);
\end{tikzpicture}
 \caption{Subnetwork $\Gamma_0^{(k)}$.}\label{net:sub}
\end{figure}

\begin{thm}  \label{thm:GZ2}
$$
{\rm im}(L\Gamma_0) = \overline{\C_2}\cap\{t_0^k=0|\;k=0,\dots,n\}.
$$
\end{thm}
Thus, the image of the map $L\Gamma_0$ coincides with the closure of the cone defined by the interlacing inequalities and describing the eigenvalues of Hermitian matrices.

\subsection{Planar networks and Horn}
We will call two planar networks $(\Gamma,[a,b])$ and
$(\Delta,[a',b'])$ {\em composable} if $a'=b$ and the set of sources
of $\Delta$ is a subset of the set of sinks of $\Gamma$. The network
$(\Gamma\circ\Delta,[a,b'])$ is then defined in the obvious manner.

\begin{defn} We say that
  a subnetwork of
  $\Gamma\circ\Delta$ is  a $\Gamma\Delta$-path if it is the union of the composable
  multipaths $\Gamma$ and $\Delta$. Such a subnetwork belongs to a set
\[  \pp ki =\{\alpha=\gamma\cup\delta|\;\gamma\in P_k\Gamma,\,
\delta\in P_i\Delta\ \text{ a composable pair}\},
\]
for some $k \geq i \geq 0$.
\end{defn}

We have the following characterization of these subnetworks.
\begin{lem} \label{lem:charpaths} A subnetwork $\alpha$ of
  $\Gamma\circ\Delta$ is a $\Gamma\Delta$-path if it has only vertices
  of degrees $(0,1)$, $(1,0)$, and $(1,1)$  and
  \begin{itemize}
  \item the vertices of degree $(0,1)$ are sources of $\Gamma$,
  \item the vertices of degree $(1,0)$ are sinks of $\Gamma$ or $\Delta$.
  \end{itemize}
\end{lem}

Using this  type of subnetworks, we can fill in the values at the nodes of
the triangulation $T=T_n$ as follows. For $n \geq k \geq i \geq 0$ and a
weighting $\epsilon:E(\Gamma\circ\Delta)\to\T$, we define the piecewise
linear function
$$
\mm ki(\epsilon) =
\max_{\alpha \in \pp ki} w_\alpha(\epsilon).
$$
Putting these together, we obtain a  piecewise linear map
  $M{\Gamma\Delta}:\T^{E(\Gamma\circ\Delta)}\to\T^{VT}$ given by
  $t^k_i=\mm ki(\epsilon)$  for $n \geq k \geq i \geq 0$.

\begin{thm} \label{thm:Horn1} Let $\Gamma$ and $\Delta$ be two
  composable networks of rank $n$. Then the
  piecewise linear map $M{\Gamma\Delta}$ satisfies
\[     \im(M{\Gamma\Delta}) \subset \overline{\C_3}\cap\{t^0_0=0\},
\]
where $\C_3$ is the polyhedral cone defined by  inequalities
\eqref{eq:KT}.
\end{thm}

Similarly to the Gelfand--Zeitlin case, we have the following
completeness result.

\begin{thm}  \label{thm:Horn2}
For planar networks $\Gamma_0$ shown in Figure \ref{net:triang} and $\Delta_0$ shown in Figure \ref{net:straight}, we have
$$
\im(M{\Gamma_0  \Delta_0}) = \overline{\C_3}\cap\{t^0_0=0\},
$$
and consequently
\[   \im(\partial\circ M{\Gamma_0\Delta_0})=\overline\CH.
\]
\end{thm}
In other words, the image of the map $\partial\circ M\Gamma_0\Delta_0$ coincides with the closure of the Horn cone describing the eigenvalues of triples of Hermitian matrices.

\section{Proofs}

\subsection{Proof of Theorem \ref{thm:GZ1}} We observe that a path in
$P_1\Gamma$ can also be described as the graph of a continuous
function $f:[a,b]\to\R$ such that
$\gr(f)=\{(x,f(x)),\,x\in[a,b]\}\subset\Gamma$. In these terms, we
can describe $P_k\Gamma$ as
\[   P_k\Gamma = \{(f_1,\dots,f_k)|\;
f_i:[a,b]\to\R\text{ continuous},\,\gr(f_i)\subset\Gamma, f_i<f_{i+1},
\,i=1,\dots,k\}.
\]
Here and below we use the notation $f<g$ as a shorthand for
$f(x)<g(x)$,  $x\in[a,b]$.  To make our notation more readable,
for a $k$-tuple $\blf=(f_1,\dots,f_k)\in P_k\Gamma$, we will write
\[      w[\blf]\text{ for the functional
}w_{\cup_{i=1}^k\gr(f_i)}:\T^{E\Gamma}\to \T.
\]

 Theorem  \ref{thm:GZ1}  is equivalent to the inequalities
\begin{equation}  \label{eq:ineqgamma}
\begin{array}{lll}
l_i\Gamma^{(k)}(\epsilon) + l_{i-1}\Gamma^{(k-1)}(\epsilon) & \geq & l_{i-1}\Gamma^{(k)}(\epsilon) +
l_{i}\Gamma^{(k-1)}(\epsilon),   \\
l_i\Gamma^{(k)}(\epsilon) + l_{i}\Gamma^{(k-1)}(\epsilon)  & \geq &l_{i+1}\Gamma^{(k)}(\epsilon) +
l_{i-1}\Gamma^{(k-1)}(\epsilon),
\end{array}
\end{equation}
for $\epsilon\in\T^{E\Gamma}$.

Consider the first of the two inequalities. Our method of proof is to show that,
for $\blf\in P_{i-1}\Gamma^{(k)}$ and $\blg\in P_i\Gamma^{(k-1)}$, there
exist $\tilde\blf\in P_{i-1}\Gamma^{(k-1)}$ and $\tilde\blg\in
P_i\Gamma^{(k)}$  such that
\begin{equation}
  \label{ftilde}
  w[\tilde\blf]+w[\tilde\blg]=w[\blf]+w[\blg].
\end{equation}

For a positive integer $N$, we consider the map
$\dor=(\dor_1,\dots,\dor_N):\R^N\to\R^N$ sending each vector to the vector with the same coordinates but listed in
decreasing order: $\sigma_1\geq\sigma_2\geq\dots\geq\sigma_N$. Thus
$\dor_1:\R^N\to\R$ is simply the largest coordinate of a
vector. Clearly, $\dor$ is a piecewise linear, continuous map. Now,
for a  planar network $\Gamma$  and $2\leq i\leq n$, we can
define a map
\[ P_{i-1}\Gamma\times P_i\Gamma\to
P_{i-1}\Gamma\times P_i\Gamma\]
in the following way. For an $(i-1)$-path $\blf=(f_1,\dots,f_{i-1})$
and an $i$-path $\blg=(g_1,\dots,g_i)$, we set
$\blf\cup\blg=(f_1,\dots,f_{i-1},g_1,\dots,g_i)$.
We observe that, for
$x\in[a,b]$, no real number can occur more than twice in the sequence
$$(f_1(x),\dots,f_{i-1}(x),g_1(x),\dots,g_i(x))$$ of length $(2i-1)$. Define the function
$\dor_j(\blf\cup\blg)$ by
$$\dor_j(\blf\cup\blg)(x)=\dor_j(f_1(x),\dots,f_{i-1}(x),g_1(x),\dots,g_i(x)).$$
We have $\dor_{j-1}(\blf\cup\blg)<\dor_{j+1}(\blf\cup\blg)$  for all $j$.
Now we can define our map by sending the pair $(\blf,\blg)$ to the pair $(\even(\blf,\blg),\odd(\blf,\blg))$, where
$
\even(\blf,\blg)=
[(\dor_2(\blf\cup\blg),\dor_4(\blf\cup\blg),\dots,\dor_{2i-2}(\blf\cup\blg))]\in
P_{i-1}\Gamma
$
and
\[   \odd(\blf,\blg)=
[(\dor_1(\blf\cup\blg),\dor_3(\blf\cup\blg),\dots,\dor_{2i-1}\blf\cup\blg))]\in
P_i\Gamma.
\]
This map has the following two properties:
\begin{itemize}
\item if $\blf\in P_{i-1}\Gamma$ and $\blg\in P_i\Gamma$, then
\[  w[\even(\blf,\blg)]+w[\odd(\blf,\blg)]=w[\blf]+w[\blg],
\], 
\item if $\blf\in P_{i-1}\Gamma^{(k)}$ and $\blg\in
  P_i\Gamma^{(k-1)}$, then
\[  \even(\blf,\blg)\in  P_{i-1}\Gamma^{(k-1)}  \text{ and }
\odd(\blf,\blg)\in P_{i}\Gamma^{(k)}.
\]
\end{itemize}
Thus, these two multipaths can be chosen as
$\tilde\blf$ and $\tilde\blg$ in \eqref{ftilde}, which implies the first
inequality in \eqref{eq:ineqgamma}.

The second inequality is proved in a similar manner. We define a map
\[ P_{i+1}\Gamma\times P_{i-1}\Gamma\to P_{i}\Gamma\times P_{i}\Gamma\]
by sending a pair of $i$-paths $(\blf,\blg)$ of length $2i$ to the pair $(\even(\blf,\blg),\odd(\blf,\blg))$, where

\[
\even(\blf,\blg)=
[(\dor_2(\blf\cup\blg),\dor_4(\blf\cup\blg),\dots,\dor_{2i}(\blf\cup\blg))]\in
P_{i}\Gamma
\]
and
\[   \odd(\blf,\blg)=
[(\dor_1(\blf\cup\blg),\dor_3(\blf\cup\blg),\dots,\dor_{2i-1}\blf\cup\blg))]\in
P_i\Gamma.
\]
As in the previous case, we observe that $
w[\even(\blf,\blg)]+w[\odd(\blf,\blg)]=w[\blf]+w[\blg]$, and if
$\blf\in P_{i+1}\Gamma^{(k)}$ and $\blg\in P_{i-1}\Gamma^{(k-1)}$,
then $\even(\blf,\blg)\in P_{i}\Gamma^{(k-1)}$ and $\odd(\blf,\blg)\in
P_{i}\Gamma^{(k)}$, which implies the second inequality in \eqref{eq:ineqgamma}.

\subsection{Proof of Theorem  \ref{thm:GZ2}}

We will prove a somewhat stronger statement: there exists a graph and
a collection of multipaths for which the image of $L\Gamma$ is the full cone $\C_{GZ}$. In this section, we
return to  the functional $w_\alpha:\T^{E\Gamma}\to\T$, for  $\alpha\in
E\Gamma$, defined as
$w_\alpha=\sum_{e\in\alpha}w_e$.

We will call a choice of multipaths $\alpha(k,i)\in P_i\Gamma^{(k)}$,
$k=1,\dots,n,\, i=1,\dots,k$ a {\em collection}. Then to each
collection $A=\{\alpha(k,i)\}_{i,k}$, we can associate the linear function $w_A:\T^{E\Gamma}\to\T^{VT}\cap\{t_0^k=0,\, k=0,\dots,n\}$ defined by $t_i^k=w_{\alpha(k,i)}$ for all $i$ and $k$.  We say that the collection $A$ {\em non-degenerate} if $w_A$ is surjective.

Let $\Gamma$ be a planar network, and let
$A=\{\alpha(k,i)\}$ be a non-degenerate collection of multipaths in
$\Gamma$. Denote by $\bint(A)\subset\T^{E\Gamma}$ the subset of those
weightings of $\Gamma$ for which the collection $A$ satisfies the
interlacing inequalities
\begin{eqnarray*} \bint(A) &=&
\{\epsilon\in\T^{E\Gamma}|\;(t^k_i=w_{\alpha(k,i)}(\epsilon),0<i\leq
k\leq n)\in\overline{\C_2}\}=\\
&=&w_A^{-1}(\overline{\C_2}\cap\{t_0^k=0|\;k=0,\dots,n\}),
\end{eqnarray*}
and let $\B_{\max}(A)$ be the subset of
weightings for which each of the multipaths in $A$ is maximal:
\[  \B_{\max}(A) = \{\epsilon\in\T^{E\Gamma}|\;
l_i\Gamma^{(k)}(\epsilon)=w_{\alpha(k,i)}(\epsilon)\}.
\] In these terms, Theorem \ref{thm:GZ1}
is equivalent to the statement that, for any
collection $A$, we have $\B_{\max}(A)\subset\bint(A)$.
  \begin{lem} \label{tautology}
    Suppose that a planar network $\Gamma$ has a
    non-degenerate collection $A$ such that
    $\B_{\max}(A)\supset\bint(A)$. Then, in fact,
    \begin{equation}
      \label{eleq}
  L\Gamma(\B_{\max}(A))=L\Gamma(\bint(A))=\overline{\C_2}\cap\{t_0^k=0|\;k=0,\dots,n\}
    \end{equation}
  \end{lem}
  \begin{proof}
  Indeed, we have
  $\bint(A)=w_A^{-1}(\overline{\C_2}\cap\{t_0^k=0|\;k=0,\dots,n\})$,
  and since $\Gamma$ is nondegenerate, we also have
  $w_A(\bint(A))=\overline{\C_2}\cap\{t_0^k=0|\;k=0,\dots,n\}$. On the
  other hand, the restrictions $L\Gamma|\B_{\max}(A)$ and $w_A|\B_{\max}(A)$ coincide, and this,   combined with $\B_{\max}(A)\supset\bint(A)$ implies \eqref{eleq}.
  \end{proof}

  Now we consider the planar network $\Gamma_0$ shown in Figure
  \ref{net:triang}. Given a decreasing sequence of integers $\ba = (a_1,\dots,a_i)$ and an increasing sequence $\bb =
  (b_1,\dots, b_i)$, we will say that a multipath $\alpha\in
  P_i\Gamma_0$ is of type $[\ba,\bb]$ if its sources are given by the
  list $\ba$ and its sinks are given by the list $\bb$. It is
  easy to verify that there is a single multipath in $ P_i\Gamma_0^{(k)}$ of type $[(k,\dots,k-i+1),(1,\dots,i)]$. Denote this multipath by $\alpha(k,i)$ and consider the collection  $A=\{\alpha(k,i),\, 0<i\leq k\leq n\}$ (see Figure \ref{net:alphas}).

\begin{figure}[h]
\begin{tikzpicture}
\draw (-1,0) -- (6.5,0);
\draw (-1,0.5) -- (6.5,0.5);
\draw (-1,1) -- (6.5,1);
\draw (-1,1.5) -- (6.5,1.5);
\draw (-1,2) -- (6.5,2);
\draw (-1,2.5) -- (6.5,2.5);
\draw (0.5,2) -- (1,1.5);
\draw (2,2) -- (2.5,1.5);
\draw (3.5,2) -- (4,1.5);
\draw (5,2) -- (5.5,1.5);
\draw (1.5,1.5) -- (2,1);
\draw (3,1.5) -- (3.5,1);
\draw (4.5,1.5) -- (5,1);
\draw (2.5,1) -- (3,0.5);
\draw (4,1) -- (4.5,0.5);
\draw (3.5,0.5) -- (4,0);
\draw (-0.5,2.5) -- (0,2);
\draw (1,2.5) -- (1.5,2);
\draw (2.5,2.5) -- (3,2);
\draw (4,2.5) -- (4.5,2);
\draw (5.5,2.5) -- (6,2);
\draw[ultra thick,red](-1,1.52) -- (3.02,1.52) -- (3.52,1.02) -- (4.02,1.02) -- (4.52,0.52) -- (6.5,0.52);
\draw[ultra thick,red](-1,1.02) -- (2.52,1.02) -- (3.02,0.52) -- (3.52,0.52) -- (4.02,0.02) -- (6.5,0.02);
\fill[black] (-1,0) circle (0pt) node[left]{1};
\fill[black] (-1,1.5) circle (0pt) node[left]{4};
\fill[black] (-1,2.5) circle (0pt) node[left]{6};
\fill[black] (6.5,0) circle (0pt) node[right]{1};
\fill[black] (6.5,0.5) circle (0pt) node[right]{2};
\fill[black] (6.5,2.5) circle (0pt) node[right]{6};
\end{tikzpicture}
 \caption{$\alpha(4,2)$ for $n=6$, shown in thick red.}\label{net:alphas}
\end{figure}
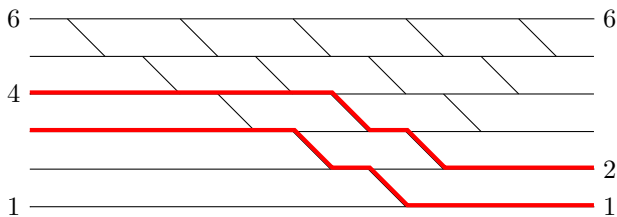

\begin{lem}\label{nondegenerate}The collection $A$ is non-degenerate.
\end{lem}
\begin{proof}
As shown in Figure \ref{net:edges}, we introduce the following notation for the weights of the edges of $\Gamma$:
\begin{itemize}
\item we denote by $h_i$ the weights of the horizontal edges adjacent to the sinks,
\item we denote by $a_{ij}$ the weights of the slanted edges,
\item the weights of the rest of the edges we put equal to 0.
\end{itemize}

\begin{figure}[h]
\begin{tikzpicture}
\draw (-1,0) -- (7.5,0);
\draw (-1,0.5) -- (7.5,0.5);
\draw (-1,1) -- (7.5,1);
\draw (-1,1.5) -- (7.5,1.5);
\draw (-1,2) -- (7.5,2);
\draw (0.5,2) -- (1,1.5);
\draw (2,2) -- (2.5,1.5);
\draw (3.5,2) -- (4,1.5);
\draw (5,2) -- (5.5,1.5);
\draw (1.5,1.5) -- (2,1);
\draw (3,1.5) -- (3.5,1);
\draw (4.5,1.5) -- (5,1);
\draw (2.5,1) -- (3,0.5);
\draw (4,1) -- (4.5,0.5);
\draw (3.5,0.5) -- (4,0);
\fill[black] (-1,0) circle (0pt) node[left]{1};
\fill[black] (-1,0.5) circle (0pt) node[left]{2};
\fill[black] (-1,1) circle (0pt);
\fill[black] (-1,1.5) circle (0pt);
\fill[black] (-1,2) circle (0pt) node[left]{n};
\fill[black] (7.5,0) circle (0pt) node[right]{1};
\fill[black] (7.5,0.5) circle (0pt) node[right]{2};
\fill[black] (7.5,1) circle (0pt);
\fill[black] (7.5,1.5) circle (0pt);
\fill[black] (7.5,2) circle (0pt) node[right]{n};
\draw[font=\small] (5.5,0.15) node{$h_1$};
\draw[font=\small] (5.7,0.65) node{$h_2$};
\draw[font=\small] (6.4,2.15) node{$h_n$};
\draw[font=\small] (4.2,0.2) node{$a_{1,1}$};
\draw[font=\small] (3.2,0.7) node{$a_{2,1}$};
\draw[font=\small] (4.7,0.7) node{$a_{2,2}$};
\draw[font=\small] (1.3,1.7) node{$a_{n-1,1}$};
\draw[font=\small] (4.3,1.7) node{$a_{n-1,2}$};
\draw[font=\small] (6,1.7) node{$a_{n-1,n-1}$};
\end{tikzpicture}
 \caption{Weighting of the edges of $\Gamma_0$.}\label{net:edges}
\end{figure}
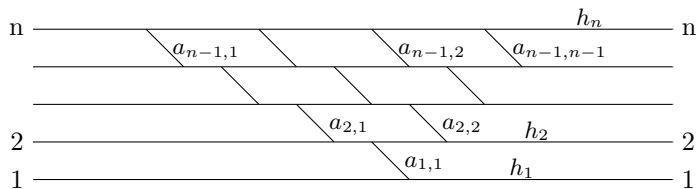

Then we can regard the  map $w_A$  as a linear map from $\T^{\frac{n(n+1)}{2}}$ to itself, and in order to prove its surjectivity it is sufficient  to show that the Jacobian of $w_A$ is nonzero.

We have $\dfrac{\partial w_{\alpha(k,i)}}{\partial h_k}=0$ for $i<k$, since these multipaths $\alpha(k,i)$ do not contain the edge weighted $h_k$ and $\dfrac{\partial w_{\alpha(k,k)}}{\partial h_k}=1$. Similarly, the first multipath of our collection $A$ containing the edge with weight $a_{r,s}$ is the multipath $\alpha(r+1,s)$, therefore we have $\dfrac{\partial w_{\alpha(k,i)}}{\partial a_{k-1,i}}=1$ and $\dfrac{\partial w_{\alpha(k,i)}}{\partial a_{r,s}}=0$ for $k\leq r$. Thus, with an appropriate ordering of the variables $h_i$ and $a_{i,j}$, the Jacobi matrix of $w_A$ is triangular with 1's on the diagonal.
\end{proof}

We prove Theorem \ref{thm:GZ1} in two steps:
\begin{itemize}
\item we begin with a graphical description of the interlacing
  inequalities describing $\bint(A)$ in terms of cells in the
  complement of $\Gamma_0$,
\item then, again, using a graphical device, we show that should these
  inequalities hold, our collection $A$ will consist of maximal multipaths.
\end{itemize}
Thus, the theorem will be proved if we show that the conditions of Lemma \ref{tautology} hold for our collection $A$. 

\subsubsection{Graphical representation of the interlacing inequalities}
We enumerate the {\em cells} formed by the connected components of the
complement of $\Gamma_0$ by the symbols $[k,i]$, $k=0,\dots, n-1$,
$i=0,\dots,k$ as shown in Figure \ref{net:triang-cells}.

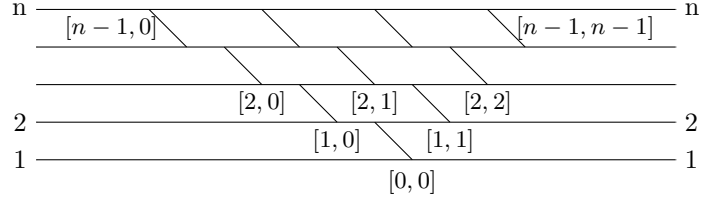
\begin{figure}[h]
\begin{tikzpicture}
\draw (-1,0) -- (7.5,0);
\draw (-1,0.5) -- (7.5,0.5);
\draw (-1,1) -- (7.5,1);
\draw (-1,1.5) -- (7.5,1.5);
\draw (-1,2) -- (7.5,2);
\draw (0.5,2) -- (1,1.5);
\draw (2,2) -- (2.5,1.5);
\draw (3.5,2) -- (4,1.5);
\draw (5,2) -- (5.5,1.5);
\draw (1.5,1.5) -- (2,1);
\draw (3,1.5) -- (3.5,1);
\draw (4.5,1.5) -- (5,1);
\draw (2.5,1) -- (3,0.5);
\draw (4,1) -- (4.5,0.5);
\draw (3.5,0.5) -- (4,0);
\fill[black] (-1,0) circle (0pt) node[left]{1};
\fill[black] (-1,0.5) circle (0pt) node[left]{2};
\fill[black] (-1,1) circle (0pt);
\fill[black] (-1,1.5) circle (0pt);
\fill[black] (-1,2) circle (0pt) node[left]{n};
\fill[black] (7.5,0) circle (0pt) node[right]{1};
\fill[black] (7.5,0.5) circle (0pt) node[right]{2};
\fill[black] (7.5,1) circle (0pt);
\fill[black] (7.5,1.5) circle (0pt);
\fill[black] (7.5,2) circle (0pt) node[right]{n};
\path (4,-0.3) node [font=\small] {$[0,0]$};
\path (3,0.25) node [font=\small] {$[1,0]$};
\path (2,0.75) node [font=\small] {$[2,0]$};
\path (0,1.75) node [font=\small] {$[n-1,0]$};
\path (4.5,0.25) node [font=\small] {$[1,1]$};
\path (3.5,0.75) node [font=\small] {$[2,1]$};
\path (5,0.75) node [font=\small] {$[2,2]$};
\path (6.3,1.75) node [font=\small] {$[n-1,n-1]$};
\end{tikzpicture}
 \caption{Enumeration of the cells of $\Gamma_0$.}\label{net:triang-cells}
\end{figure}

To the cell $[k,i]$, we associate the functional
$c_{[k,i]}:\T^{E\Gamma_0}\to\T$ that is the sum of the signed weights along
the clockwise oriented boundary of the cell, where the sign depends
on  whether the orientation of the boundary coincides with the
orientation of the edge or not (Figure \ref{cell}). The same applies to the unbounded cells.

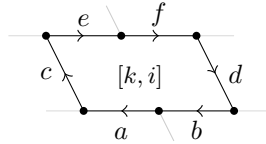
\begin{figure}[h]
\begin{tikzpicture}
\fill[black] (0,0) circle (1.5pt);
\fill[black] (1,0) circle (1.5pt);
\fill[black] (2,0) circle (1.5pt);
\fill[black] (1.5,1) circle (1.5pt);
\fill[black] (0.5,1) circle (1.5pt);
\fill[black] (-0.5,1) circle (1.5pt);
\tikzset{->-/.style={decoration={markings, mark=at position .5 with {\arrow[scale=1.3]{>}}},postaction={decorate}}}
\draw[->-] (1,0) -- (0,0);
\draw[->-] (2,0) -- (1,0);
\draw[->-] (1.5,1) -- (2,0);
\draw[->-] (0.5,1) -- (1.5,1);
\draw[->-] (-0.5,1) -- (0.5,1);
\draw[->-] (0,0) -- (-0.5,1);
\draw[opacity=0.15] (0,0) -- (-0.5,0);
\draw[opacity=0.2] (1,0) -- (1.2,-0.4);
\draw[opacity=0.15] (2,0) -- (2.5,0);
\draw[opacity=0.15] (1.5,1) -- (2,1);
\draw[opacity=0.2] (0.5,1) -- (0.3,1.4);
\draw[opacity=0.15] (-0.5,1) -- (-1,1);
\draw (0.5,-0.1) node[below]{$a$};
\draw (1.5,0) node[below]{$b$};
\draw (1.8,0.5) node[right]{$d$};
\draw (-0.3,0.5) node[left]{$c$};
\draw (0,1) node[above]{$e$};
\draw (1,1) node[above]{$f$};
\draw (0.75,0.45) node[font=\small]{$[k,i]$};
\end{tikzpicture}\caption{$c_{[k,i]}=-w_a-w_b-w_c+w_e+w_f+w_d$.}\label{cell}
\end{figure}

Next, we define the functionals
\[ r^\nearrow_{[k,i]} =c_{[k,i]}+c_{[k-1,i-1]}+\dots+c_{[k-i,0]}
\]
and
\[r^\searrow_{[k,i]} =c_{[k,i]}+c_{[k-1,i]}+\dots+c_{[i,i]},
\]
which correspond to the boundary of the shaded regions shown in
Figures \ref{net:regions-up} and \ref{net:regions-down} below.

\begin{figure}[h]
\begin{tikzpicture}
\draw (-1,3) -- (6.5,3);
\draw (3,2.75) node{$\dots\dots$};
\draw (-1,0) -- (6.5,0);
\draw (-1,0.5) -- (6.5,0.5);
\draw (-1,1) -- (6.5,1);
\draw (-1,1.5) -- (6.5,1.5);
\draw (-1,2) -- (6.5,2);
\draw (-1,2.5) -- (6.5,2.5);
\draw (0.5,2) -- (1,1.5);
\draw (2,2) -- (2.5,1.5);
\draw (3.5,2) -- (4,1.5);
\draw (5,2) -- (5.5,1.5);
\draw (1.5,1.5) -- (2,1);
\draw (3,1.5) -- (3.5,1);
\draw (4.5,1.5) -- (5,1);
\draw (2.5,1) -- (3,0.5);
\draw (4,1) -- (4.5,0.5);
\draw (2.8,0.25) node{$\dots\dots$};
\draw (-0.5,2.5) -- (0,2);
\draw (1,2.5) -- (1.5,2);
\draw (2.5,2.5) -- (3,2);
\draw (4,2.5) -- (4.5,2);
\draw (5.5,2.5) -- (6,2);
\draw[fill=gray,fill opacity=0.15, line width=0.01] (-1,0.5) -- (3,0.5) -- (2.5,1) -- (-1,1) -- cycle;
\draw[fill=gray,fill opacity=0.15, line width=0.01] (2,1) -- (3.5,1) -- (3,1.5) -- (1.5,1.5) -- cycle;
\draw[fill=gray,fill opacity=0.15, line width=0.01] (2.5,1.5) -- (4,1.5) -- (3.5,2) -- (2,2) -- cycle;
\draw[fill=gray,fill opacity=0.15, line width=0.01] (3,2) -- (4.5,2) -- (4,2.5) -- (2.5,2.5) -- cycle;
\fill[black] (-1,0) circle (0pt) node[left]{1};
\fill[black] (-1,0.5) circle (0pt) node[left]{$k-i$};
\fill[black] (-1,1) circle (0pt) node[left]{};
\fill[black] (-1,3) circle (0pt) node[left]{$n$};
\fill[black] (-1,2.5) circle (0pt) node[left]{$k+1$};
\fill[black] (6.5,0) circle (0pt) node[right]{1};
\fill[black] (6.5,0.5) circle (0pt) node[right]{$k-i$};
\fill[black] (6.5,1) circle (0pt) node[right]{};
\fill[black] (6.5,3) circle (0pt) node[right]{$n$};
\fill[black] (6.5,2.5) circle (0pt) node[right]{$k+1$};
\path (3.5,2.25) node [font=\small] {$[k,i]$};
\path (2,0.75) node [font=\small] {$[k-i,0]$};
\end{tikzpicture}
 \caption{$r^\nearrow_{[k,i]}$.}\label{net:regions-up}
\end{figure}

\begin{figure}[h]
\begin{tikzpicture}
\draw (-1,3) -- (6.5,3);
\draw (2.8,2.75) node{$\dots\dots$};
\draw (-1,0) -- (6.5,0);
\draw (-1,0.5) -- (6.5,0.5);
\draw (-1,1) -- (6.5,1);
\draw (-1,1.5) -- (6.5,1.5);
\draw (-1,2) -- (6.5,2);
\draw (-1,2.5) -- (6.5,2.5);
\draw (0.5,2) -- (1,1.5);
\draw (2,2) -- (2.5,1.5);
\draw (3.5,2) -- (4,1.5);
\draw (5,2) -- (5.5,1.5);
\draw (1.5,1.5) -- (2,1);
\draw (3,1.5) -- (3.5,1);
\draw (4.5,1.5) -- (5,1);
\draw (2.5,1) -- (3,0.5);
\draw (4,1) -- (4.5,0.5);
\draw (2.8,0.25) node{$\dots\dots$};
\draw (-0.5,2.5) -- (0,2);
\draw (1,2.5) -- (1.5,2);
\draw (2.5,2.5) -- (3,2);
\draw (4,2.5) -- (4.5,2);
\draw (5.5,2.5) -- (6,2);
\draw[fill=gray,fill opacity=0.15, line width=0.01] (5,1) -- (6.5,1) -- (6.5,1.5) -- (4.5,1.5) -- cycle;
\draw[fill=gray,fill opacity=0.15, line width=0.01] (4,1.5) -- (5.5,1.5) -- (5,2) -- (3.5,2) -- cycle;
\draw[fill=gray,fill opacity=0.15, line width=0.01] (3,2) -- (4.5,2) -- (4,2.5) -- (2.5,2.5) -- cycle;
\fill[black] (-1,0) circle (0pt) node[left]{1};
\fill[black] (-1,0.5) circle (0pt) node[left]{};
\fill[black] (-1,1) circle (0pt) node[left]{$i$};
\fill[black] (-1,3) circle (0pt) node[left]{$n$};
\fill[black] (-1,2.5) circle (0pt) node[left]{$k+1$};
\fill[black] (6.5,0) circle (0pt) node[right]{1};
\fill[black] (6.5,0.5) circle (0pt) node[right]{};
\fill[black] (6.5,1) circle (0pt) node[right]{$i$};
\fill[black] (6.5,3) circle (0pt) node[right]{$n$};
\fill[black] (6.5,2.5) circle (0pt) node[right]{$k+1$};
\path (3.5,2.25) node [font=\small] {$[k,i]$};
\path (5.5,1.25) node [font=\small] {$[i,i]$};
\end{tikzpicture}
 \caption{$r^\searrow_{[k,i]}$.}\label{net:regions-down}
\end{figure}

\begin{lem}
  \label{relgraph}
 The polyhedral cone $\bint(A)$ is given by the weightings
 $\epsilon\in\T^{E\Gamma_0}$
satisfying
   \begin{equation} \label{bintineqs}
     \begin{array}{ll}
r^\nearrow_{[k,i]}(\epsilon)\geq  0,\; & 0\leq i <  k<n,
\\
r^\searrow_{[k,i]}(\epsilon)\leq  0, \; & 0< i\leq   k < n.
     \end{array}
   \end{equation}
\end{lem}

\begin{proof}
By  Definition \ref{defc2} that $\bint(A)$ is defined by the inequalities
 \[        w_{\alpha(k+1,i)}+
w_{\alpha(k,i-1)}-w_{\alpha(k+1,i-1)}-w_{\alpha(k,i)}\geq 0
\]
and
\[       w_{\alpha(k+1,i+1)}+w_{\alpha(k,i-1)}- w_{\alpha(k+1,i)}-w_{\alpha(k,i)}\leq0,
\]
for $0<i\leq k<n$.

It is not difficult to identify these inequalities with  inequalities
\eqref{bintineqs}.  In Figures \ref{net:regterm1} and \ref{net:regterm2}, we present a graphical proof of this equivalence for the case of the first inequality in \eqref{bintineqs} when
 $k+1=n=5$ and $i=2$:
\begin{itemize}
\item the thick (red) lines show $\alpha(k+1,i)$ and $\alpha(k,i)$,
\item the dashed lines show $\alpha(k+1,i-1)$ and $\alpha(k,i-1)$,
\item the sum of weights along the boundary of the shaded area
  represents the difference in the caption.
\end{itemize}
One then notes that the difference of the two shaded regions is the
union of the cells  $[k,i],\,\dots,\,[k-i,0]$, which, in view of the
definition of $r^\nearrow_{[k,i]}$, completes the proof. The proof of
the equivalence of the other pair of inequalities is analogous.
\begin{figure}[h]
\begin{minipage}{0.48\linewidth}
\begin{tikzpicture}[scale=0.8]
\draw (0,0) -- (6,0);
\draw (0,0.5) -- (6,0.5);
\draw (0,1) -- (6,1);
\draw (0,1.5) -- (6,1.5);
\draw (0,2) -- (6,2);
\draw (0.5,2) -- (1,1.5);
\draw (2,2) -- (2.5,1.5);
\draw (3.5,2) -- (4,1.5);
\draw (5,2) -- (5.5,1.5);
\draw (1.5,1.5) -- (2,1);
\draw (3,1.5) -- (3.5,1);
\draw (4.5,1.5) -- (5,1);
\draw (2.5,1) -- (3,0.5);
\draw (4,1) -- (4.5,0.5);
\draw (3.5,0.5) -- (4,0);
\draw[very thick,red] (0,2) -- (2,2) -- (2.5,1.5) -- (3,1.5) -- (3.5,1) -- (4,1) -- (4.5,0.5) -- (6,0.5);
\draw[very thick,red] (0,1.5) -- (1.5,1.5) -- (2,1) -- (2.5,1) -- (3,0.5) -- (3.5,0.5) -- (4,0) -- (6,0);
\draw[very thick,dashed,blue] (0,1.97) -- (0.47,1.97) -- (0.97,1.47) -- (1.47,1.47) -- (1.97,0.97) -- (2.47,0.97) -- (2.97,0.47) -- (3.47,0.47) -- (3.97,-0.03) -- (6,-0.03);
\draw[fill=gray,fill opacity=0.1, line width=0.01] (0,1.5) -- (1.5,1.5) -- (2,1) -- (2.5,1) -- (3,0.5) -- (3.5,0.5) -- (4,0) -- (6,0) -- (6,-0.5) -- (0,-0.5) -- cycle;
\draw[fill=gray,fill opacity=0.1, line width=0.01] (0.5,2) -- (2,2) -- (2.5,1.5) -- (3,1.5) -- (3.5,1) -- (4,1) -- (4.5,0.5) -- (6,0.5)  -- (6,0) -- (4,0) -- (3.5,0.5) --  (3,0.5) -- (2.5,1) -- (2,1) -- (1.5,1.5) -- (1,1.5)  -- cycle;
\fill[black] (0,0) circle (0pt) node[left, font=\tiny]{1};
\fill[black] (0,0.5) circle (0pt) node[left, font=\tiny]{2};
\fill[black] (0,1) circle (0pt) node[left, font=\tiny]{3};
\fill[black] (0,1.5) circle (0pt) node[left, font=\tiny]{4};
\fill[black] (0,2) circle (0pt) node[left, font=\tiny]{5};
\fill[black] (6,0) circle (0pt) node[right, font=\tiny]{1};
\fill[black] (6,0.5) circle (0pt) node[right, font=\tiny]{2};
\fill[black] (6,1) circle (0pt) node[right, font=\tiny]{3};
\fill[black] (6,1.5) circle (0pt) node[right, font=\tiny]{4};
\fill[black] (6,2) circle (0pt) node[right, font=\tiny]{5};
\end{tikzpicture}
 \caption{\small$w_{\alpha(k+1,i)}-w_{\alpha(k+1,i-1)}$.}\label{net:regterm1}
\end{minipage}
\begin{minipage}{0.48\linewidth}
\begin{tikzpicture}[scale=0.8]
\draw (0,0) -- (6,0);
\draw (0,0.5) -- (6,0.5);
\draw (0,1) -- (6,1);
\draw (0,1.5) -- (6,1.5);
\draw (0,2) -- (6,2);
\draw (0.5,2) -- (1,1.5);
\draw (2,2) -- (2.5,1.5);
\draw (3.5,2) -- (4,1.5);
\draw (5,2) -- (5.5,1.5);
\draw (1.5,1.5) -- (2,1);
\draw (3,1.5) -- (3.5,1);
\draw (4.5,1.5) -- (5,1);
\draw (2.5,1) -- (3,0.5);
\draw (4,1) -- (4.5,0.5);
\draw (3.5,0.5) -- (4,0);
\draw[very thick,red] (0,1.5) -- (3,1.5) -- (3.5,1) -- (4,1) -- (4.5,0.5) -- (6,0.5);
\draw[very thick,red] (0,1) -- (2.5,1) -- (3,0.5) -- (3.5,0.5) -- (4,0) -- (6,0);
\draw[dashed,blue,very thick] (0,1.47) -- (1.47,1.47) -- (1.97,0.97) -- (2.47,0.97) -- (2.97,0.47) -- (3.47,0.47) -- (3.97,-0.03) -- (6,-0.03);
\draw[fill=gray,fill opacity=0.1, line width=0.01] (0,1) -- (2,1) -- (2.5,1) -- (3,0.5) -- (3.5,0.5) -- (4,0) -- (6,0) -- (6,-0.5) -- (0,-0.5) -- cycle;
\draw[fill=gray,fill opacity=0.1, line width=0.01] (1.5,1.5) -- (2.5,1.5) -- (3,1.5) -- (3.5,1) -- (4,1) -- (4.5,0.5) -- (6,0.5)  -- (6,0) -- (4,0) -- (3.5,0.5) --  (3,0.5) -- (2.5,1) -- (2,1)  -- cycle;
\fill[black] (0,0) circle (0pt) node[left, font=\tiny]{1};
\fill[black] (0,0.5) circle (0pt) node[left, font=\tiny]{2};
\fill[black] (0,1) circle (0pt) node[left, font=\tiny]{3};
\fill[black] (0,1.5) circle (0pt) node[left, font=\tiny]{4};
\fill[black] (0,2) circle (0pt) node[left, font=\tiny]{5};
\fill[black] (6,0) circle (0pt) node[right, font=\tiny]{1};
\fill[black] (6,0.5) circle (0pt) node[right, font=\tiny]{2};
\fill[black] (6,1) circle (0pt) node[right, font=\tiny]{3};
\fill[black] (6,1.5) circle (0pt) node[right, font=\tiny]{4};
\fill[black] (6,2) circle (0pt) node[right, font=\tiny]{5};
\end{tikzpicture}
 \caption{\small$w_{\alpha(k,i)}-w_{\alpha(k,i-1)}$.}\label{net:regterm2}
\end{minipage}
\end{figure}
\end{proof}

\subsubsection{Maximality of the collection $A$}
Now we are ready to prove Theorem \ref{thm:GZ2}.
We will call $\beta\in P_i\Gamma_0$ a {\em maximal multipath} for the
weighting $\epsilon$ if $l_i\Gamma_0=w_\beta(\epsilon)$.
By Lemmas
\ref{tautology} and \ref{relgraph}, we may assume that we are given a
weighting $\epsilon$ of $\Gamma_0$ satisfying
$r^\nearrow_{[k,i]}(\epsilon)\geq0$ and
$r^\searrow_{[k,i]}(\epsilon)\leq0$. We must prove that for
this weighting, the paths of the collection $A$ are maximal. We fix this $\epsilon$ for the rest of the section and will often drop it from our notation  for brevity.

For any $k$-tuple $\ba=(a_1,\dots,a_k)$, we will denote the sum $a_1+\dots+a_k$ by $\sum\ba$.

\begin{lem}\label{sinks}
  Let $\beta\in P_i\Gamma_0$ be a maximal multipath of type $[\ba,\bb]$ for
  the weighting $\epsilon\in\bint$ such that the value of
  $\sum\bb$ is minimal among all maximal multipaths
  in $P_i\Gamma_0$. Then $\bb=(1,2,\dots,i)$.
\end{lem}
\begin{proof}
  First consider the case $i=1$. If $\beta$ ends at $b>1$, then there
  is a cell $[b-1,b-1]$ whose upper edge is a part of $\beta$. Then
  $w_{\beta}\leq w_{\tilde\beta}=w_{\beta}-r^\searrow_{[k,b-1]}$ for some $k$,
  where $\tilde\beta$ is a path ending at $b-1$ (see Figure
  \ref{net:sinks-down}). This contradicts the assumption on $\beta$,
 and so the $i=1$ case is proved.

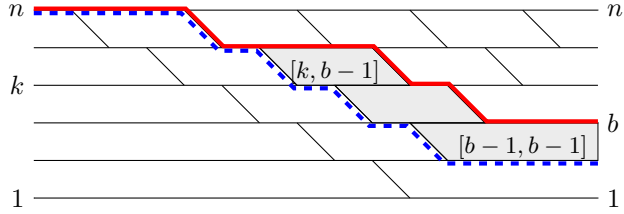
\begin{figure}[h]
\begin{tikzpicture}
\draw (-1,0) -- (6.5,0);
\draw (-1,0.5) -- (6.5,0.5);
\draw (-1,1) -- (6.5,1);
\draw (-1,1.5) -- (6.5,1.5);
\draw (-1,2) -- (6.5,2);
\draw (-1,2.5) -- (6.5,2.5);
\draw (0.5,2) -- (1,1.5);
\draw (2,2) -- (2.5,1.5);
\draw (3.5,2) -- (4,1.5);
\draw (5,2) -- (5.5,1.5);
\draw (1.5,1.5) -- (2,1);
\draw (3,1.5) -- (3.5,1);
\draw (4.5,1.5) -- (5,1);
\draw (2.5,1) -- (3,0.5);
\draw (4,1) -- (4.5,0.5);
\draw (3.5,0.5) -- (4,0);
\draw (-0.5,2.5) -- (0,2);
\draw (1,2.5) -- (1.5,2);
\draw (2.5,2.5) -- (3,2);
\draw (4,2.5) -- (4.5,2);
\draw (5.5,2.5) -- (6,2);
\draw[fill=gray,fill opacity=0.15, line width=0.01] (4.5,0.5) -- (6.5,0.5) -- (6.5,1) -- (4,1) -- cycle;
\draw[fill=gray,fill opacity=0.15, line width=0.01] (3.5,1) -- (5,1) -- (4.5,1.5) -- (3,1.5) -- cycle;
\draw[fill=gray,fill opacity=0.15, line width=0.01] (2.5,1.5) -- (4,1.5) -- (3.5,2) -- (2,2) -- cycle;
\fill[black] (-1,0) circle (0pt) node[left]{1};
\fill[black] (-1,2.5) circle (0pt) node[left]{$n$};
\fill[black] (-1,1.5) circle (0pt) node[left]{$k$};
\fill[black] (6.5,0) circle (0pt) node[right]{1};
\fill[black] (6.5,1) circle (0pt) node[right]{$b$};
\fill[black] (6.5,2.5) circle (0pt) node[right]{$n$};
\draw(5.5,0.4) node[above,font=\small]{$[b-1,b-1]$};
\draw(3,1.4) node[above,font=\small]{$[k,b-1]$};
\draw[ultra thick,red] (-1,2.52) -- (1.02,2.52) -- (1.52,2.02) -- (3.52,2.02) -- (4.02,1.52) -- (4.52,1.52) -- (5.02,1.02) -- (6.5,1.02);
\draw[ultra thick,dashed,blue] (-1,2.46) -- (0.96,2.46) -- (1.46,1.96) -- (1.96,1.96) -- (2.46,1.46) -- (2.96,1.46) -- (3.46,0.96) -- (3.96,0.96) -- (4.46,0.46) -- (6.5,0.46);
\end{tikzpicture}
 \caption{$w_{\beta}$ (solid) $=$ $w_{\tilde\beta}$ (dashed) $-$ $r^\searrow_{[k,b-1]}$ (shaded).}\label{net:sinks-down}
\end{figure}

Now let $\beta\in P_i\Gamma_0$ be a maximal multipath for $\epsilon$
with the smallest possible value of $b_1+\dots+b_i$. Clearly, we can apply the
above argument to the lowest path in $\beta$ and conclude that
$b_1=1$. If $b_2>2$, then, as before, we can subtract the appropriate
$r^\searrow_{[k,b_2-1]}$ from $w_\beta$, and obtain a multipath with a strictly
lower sum of end-values. This contradicts the assumption on
$\beta$, hence $b_2=2$. Repeating this argument, we can show that
$b_j=j$ for $j=1,\dots,i$, which completes the proof.
\end{proof}

Now we consider the sources of maximal multipaths.

\begin{lem}\label{sources}
  Let $\beta\in P_i\Gamma_0$ be a maximal multipath of type $[\ba,(1,2,\dots,i)]$ for
  a weighting $\epsilon\in\bint$ for which the value of
  $\sum\ba=a_1+\dots+a_i$ is the largest among all the maximal multipaths
  in $P_i\Gamma_0$. Then $\ba=(n,n-1,\dots,n-i+1)$.
\end{lem}
\begin{proof}
  We begin with the case $i=1$. Suppose that a path $\beta\in
  P_1\Gamma_0$ connecting $a<n$ with the sink 1 satisfies the
  conditions of the lemma. Then $w_{\beta}\leq w_{\tilde\beta}=
  w_{\beta}+r^\nearrow_{[a,0]}$  for a
  path $\tilde\beta$ that starts at $a+1$. This contradicts our
  assumptions on $\beta$, and thus $a=n$.

  If $i>1$ and $\beta\in P_i\Gamma_0$ satisfies the conditions of the
  lemma, then we can still show as above that necessarily
  $a_1=n$.  Now assume that there is a gap between the sources of
  $\beta$, i.e. that, for some $j<n$, $j$ is a term of the sequence $\ba$, but $j+1$ is not. If the first vertex $v_{j+1,1}$ of $\Gamma_0$ on the 
  line $y=j+1$ is not in $\beta$, then we can again add
  $r^\nearrow_{[j,0]}$ to $w_\beta$ and obtain a new multipath
  with larger $\sum\ba$.

If $v_{j+1,1}\in\beta$, then $\beta$ {\em has an elbow} at
$v_{j+1,1}$, by which we mean that the two edges in $\beta$ containing
$v_{j+1,1}$ are the downward incoming edge and the horizontal outgoing
edge. The other two edges are excluded by our assumptions. Now consider
the sequence of vertices $v_{j+1,1},v_{j+2,2},\dots,v_{j+s,s}\in\beta$
such that $v_{j+s+1,s+1}\notin\beta$. It is easy to verify that $\beta$ has necessarily an elbow at all of the vertices of this
sequence.
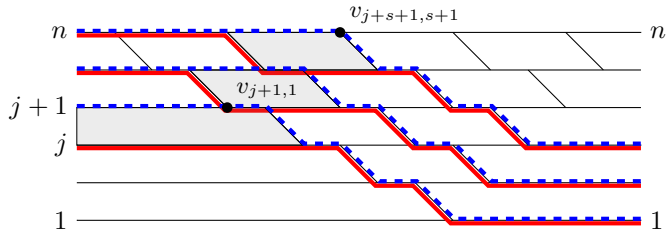
\begin{figure}[h]
\begin{tikzpicture}
\draw (-1,0) -- (6.5,0);
\draw (-1,0.5) -- (6.5,0.5);
\draw (-1,1) -- (6.5,1);
\draw (-1,1.5) -- (6.5,1.5);
\draw (-1,2) -- (6.5,2);
\draw (-1,2.5) -- (6.5,2.5);
\draw (0.5,2) -- (1,1.5);
\draw (2,2) -- (2.5,1.5);
\draw (3.5,2) -- (4,1.5);
\draw (5,2) -- (5.5,1.5);
\draw (1.5,1.5) -- (2,1);
\draw (3,1.5) -- (3.5,1);
\draw (4.5,1.5) -- (5,1);
\draw (2.5,1) -- (3,0.5);
\draw (4,1) -- (4.5,0.5);
\draw (3.5,0.5) -- (4,0);
\draw (-0.5,2.5) -- (0,2);
\draw (1,2.5) -- (1.5,2);
\draw (2.5,2.5) -- (3,2);
\draw (4,2.5) -- (4.5,2);
\draw (5.5,2.5) -- (6,2);
\fill[black] (-1,0) circle (0pt) node[left]{1};
\fill[black] (-1,2.5) circle (0pt) node[left]{$n$};
\fill[black] (-1,1.5) circle (0pt) node[left]{$j+1$};
\fill[black] (-1,1) circle (0pt) node[left]{$j$};
\fill[black] (6.5,0) circle (0pt) node[right]{1};
\fill[black] (6.5,1) circle (0pt) node[right]{};
\fill[black] (6.5,2.5) circle (0pt) node[right]{$n$};
\draw[ultra thick,red] (-1,2.46) -- (0.97,2.46) -- (1.47,1.96) -- (3.47,1.96) -- (3.97,1.46) -- (4.47,1.46) -- (4.97,0.96) -- (6.5,0.96);
\draw[ultra thick,red] (-1,1.96) -- (0.47,1.96) -- (0.97,1.46) -- (2.97,1.46) -- (3.47,0.96) -- (3.97,0.96) -- (4.47,0.46) -- (6.5,0.46);
\draw[ultra thick,red] (-1,0.96) -- (2.47,0.96) -- (2.97,0.46) -- (3.47,0.46) -- (3.97,-0.04) -- (6.5,-0.04);
\draw[fill=gray,fill opacity=0.15, line width=0.01] (-1,1) -- (2,1) -- (1.5,1.5) -- (-1,1.5) -- cycle;
\draw[fill=gray,fill opacity=0.15, line width=0.01] (1,1.5) -- (2.5,1.5) -- (2,2) -- (0.5,2) -- cycle;
\draw[fill=gray,fill opacity=0.15, line width=0.01] (1.5,2) -- (3,2) -- (2.5,2.5) -- (1,2.5) -- cycle;
\draw[dashed,ultra thick,blue] (-1,2.52) -- (2.52,2.52) -- (3.02,2.02) -- (3.52,2.02) -- (4.02,1.52) -- (4.52,1.52) -- (5.02,1.02) -- (6.5,1.02);
\draw[dashed,ultra thick,blue] (-1,2.02) --  (2.02,2.02) -- (2.52,1.52) -- (3.02,1.52) --  (3.52,1.02) -- (4.02,1.02) -- (4.52,0.52) -- (6.5,0.52);
\draw[dashed,ultra thick,blue] (-1,1.52) -- (1.52,1.52) -- (2.02,1.02) -- (2.52,1.02) -- (3.02,0.52) -- (3.52,0.52) -- (4.02,0.02) -- (6.5,0.02);
\fill[black] (1,1.5) circle (2pt) node[above right,font=\small]{$v_{j+1,1}$};
\fill[black] (2.5,2.5) circle (2pt) node[above right,font=\small]{$v_{j+s+1,s+1}$};
\end{tikzpicture}
 \caption{Graphical representation of $w_{\tilde\beta}=
w_{\beta}+r^\nearrow_{[j+s-1,s-1]}$.}\label{net:sources-up}
\end{figure}

This implies that $w_{\beta}\leq w_{\tilde\beta}=
w_{\beta}+r^\nearrow_{[j+s-1,s-1]}$ for a multipath $\tilde\beta$ with
the source at height $j$ moved up to $j+1$.  We give an example of
this operation in Figure \ref{net:sources-up}, where
\begin{itemize}
\item $i=3$ and $s=2$,
\item $\beta$ is marked with thick (red) lines,
\item $\tilde\beta$ is marked with dashed lines,
\item the shaded area is the domain corresponding to $r^\nearrow_{[j+s-1,s-1]}$.
\end{itemize}

This contradicts our
assumption that $\beta$ is a maximal multipath with the largest
$\sum\ba$, thus there can be no gaps among the sources of $\beta$, which completes the proof.

\end{proof}

Now we can quickly finish the proof of Theorem \ref{thm:GZ2}. We must show that
$\im(L{\Gamma_0})=\overline{\C_2}\cap\{t_0^k=0|\;k=0,\dots,n\}$. We consider the non-degenerate collection
$A=\{\alpha(k,i)\}$. According to Lemma \ref{tautology}, it is sufficient to show that
if the weights
$w_{\alpha(k,i)}(\epsilon)$ satisfy the interlacing inequalities for a weighting $\epsilon$ of $\Gamma_0$, then they
are also maximal. This obviously follows from Lemmas \ref{sinks} and \ref{sources}.

\subsection{Proof of Theorem  \ref{thm:Horn1}}

Let $\Gamma$ and $\Delta$  be two planar networks of rank $n$. We must prove the inequalities

\begin{equation}  \label{eq:GD}
\begin{array}{lll}
\mm{k}{i} + \mm{k-1}{i-1} & \geq & \mm{k}{i-1} + \mm{k-1}{i},  \\
\mm{k}{i} + \mm{k-1}{i} & \geq & \mm{k}{i+1} + \mm{k-1}{i-1},\\
\mm{k}{i} + \mm{k}{i-1} & \geq & \mm{k+1}{i} + \mm{k-1}{i-1} .
\end{array}
\end{equation}

We begin with the first inequality. As in the proof
of Theorem \ref{thm:GZ1}, we will show that for $\alpha\in
\pp{k-1}i$ and $\beta\in\pp k{i-1}$, one can find $\tilde \alpha\in
\pp{k-1}{i-1}$ and $\tilde \beta\in\pp ki$ such that
\[
w_\alpha(\epsilon)+w_\beta(\epsilon)=
w_{\tilde \alpha}(\epsilon)+w_{\tilde \beta}(\epsilon), \ \mathrm{for\ any\ weighting}\  \epsilon\in\T^{E(\Gamma\Delta)}.
\]
\emph{We can informally describe this problem as follows. Imagine that we
have a group of tourists with $k$ men and $k-1$ women;
$i-1$ of the men and  $i$ of the women are fit. The group is planning
an excursion where the fit tourists go from one town to a neighboring town and the unfit tourists only go from the first town to a park halfway to the second town. The organizer of the excursion devises a route for each member of the group under the following special condition: the paths of the men
should not intersect, and similarly, the paths of the women should not
intersect. Just before the trip  it turns out that
there will be $i$ fit men, and $i-1$ fit women, with the total numbers
of men and women unchanged. Can the organizer redraw the routes under
the same special condition, and so that if a certain segment was used
only by one person in the original plan, then this segment will be
used by only one person in the new plan as well?}

To prove the theorem it will be convenient to consider a slightly more general notion of a
planar network in which we allow multiple edges (in fact, we will need only double edges). If we have
a double edge emanating from a vertex, then this edge will contribute
2 to the out-degree of this vertex; similarly, this edges will
contribute 2 to the in-degree of the vertex to which it points. We
will also consider planar networks with a different number of sources
and sinks. We will call such a planar network {\em of type
  $[k_1,k_2]$} if the sum of out-degrees of all its sources is $k_1$
and the sum of all in-degrees of its sinks is $k_2$.

The sum of multipaths $\Theta=\alpha\cup\beta$ is naturally such a
generalized planar network of type $[2k-1,2i-1]$ with the edges $e\in
E\alpha\cap E\beta$ having multiplicity 2. We will consider such a
double edge as two separate edges. The end-line of $\Gamma$, which is
also the start-line of $\Delta$, will be called the {\em middle line}
of $\Theta$. According to our strategy, the inequality
\begin{equation}
  \label{eqm}
  \mm{k}{i} + \mm{k-1}{i-1}  \geq  \mm{k}{i-1} + \mm{k-1}{i}
\end{equation}
will follow if we can decompose $\Theta$ as the union of
multipaths. one from $\pp{k-1}{i-1}$ and the other from $\pp ki$. This will be
shown in Proposition \ref{hornprop} below.

More generally, we consider decompositions of $\Theta$ into two
multipaths from $\pp pq$ and $\pp {p'}{q'}$ for some integers
$p,q,p'$ and $q'$. Clearly, we will have $q+q'=2i-1$ and $p+p'=2k-1$. To such a decomposition, we can associate a coloring of the edges of $\Theta$ in two colors, say, red and
green. We will call such a coloring {\em valid}. Let us classify all valid colorings.

 First we note that just as in the proof of Lemma \ref{subnet},
the operation of eliminating a vertex of degree $(1,1)$ and replacing
its two adjacent edges by a single edge will not influence the image of $M\Gamma\Delta$, so we
will assume that $\Theta$ has no vertices of degree $(1,1)$.

Recall that a closed path in an unoriented graph is a connected
subgraph whose every vertex has degree 2; an open path has, in
addition, two vertices of degree 1.
\begin{lem}
  Assume that $\Theta$ has no vertices of degree $(1,1)$ and consider
  the equivalence relation on the edges of $\Theta$ generated by the
  following relation: two edges are related if they are either both
  incoming edges of a vertex of $\Theta$ or both outgoing edges of a
  vertex of $\Theta$.  Then the resulting equivalence classes are
  (possibly closed) unoriented paths in $\Theta$ with edges having alternating orientations.
\end{lem}

We will call this decomposition of $\Theta$ the {\em canonical path
  decomposition of $\Theta$} (see Figure \ref{net:decomp}).

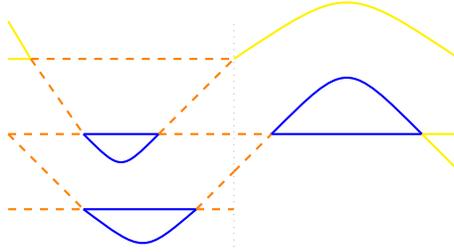
\begin{figure}[h]
\begin{tikzpicture}
\draw[opacity=0.3,dotted] (3,-0.5) -- (3,2.5);
\draw [orange,thick,dashed](0,0) -- (1,0);
\draw [blue,thick](1,0) -- (2.5,0);
\draw [orange,thick,dashed](2.5,0) -- (3,0);
\draw [orange,thick,dashed] (0,1) -- (1,1);
\draw [blue,thick](1,1) -- (2,1);
\draw [orange,thick,dashed](2,1) -- (3.5,1);
\draw [blue,thick](3.5,1) -- (5.5,1);
\draw [yellow,thick](5.5,1) -- (6,1);
\draw [yellow,thick] (0,2) -- (0.3,2);
\draw [orange,thick,dashed](0.3,2) -- (3,2);
\draw [yellow,thick](3,2) .. controls (4.5,3) .. (6,2);
\draw [orange,thick,dashed](0,1) -- (1,0);
\draw [blue,thick](1,0) .. controls (1.8,-0.6)  .. (2.5,0);
\draw [orange,thick,dashed](2.5,0) -- (3,0.5);
\draw [orange,thick,dashed](3,0.5) -- (3.5,1) ;
\draw [blue,thick](3.5,1) .. controls (4.5,2) .. (5.5,1);
\draw [yellow,thick](5.5,1) -- (6,0.5);
\draw [yellow,thick](0,2.5)  -- (0.3,2)  ;
\draw [orange,thick,dashed](0.3,2) -- (1,1);
\draw [blue,thick](1,1) .. controls (1.5,0.5) .. (2,1);
\draw [orange,thick,dashed](2,1) -- (3,2);
\end{tikzpicture}\caption{Canonical path decomposition of $\Theta$ of type $[5,3]$ ($k=3$ and $i=2$).} \label{net:decomp}
\end{figure}

\begin{proof}
The fact that $\Theta$ is the
  union of two multipaths implies that
\begin{itemize}
\item  the  sources of $\Theta$ are of degree $(0,1)$ or $(0,2)$,
\item the degrees of all the other vertices have degrees from
  the following list:
\[   (2,2),\;(2,1),\;(2,0),\;(1,1),\;(1,0);
\]
\item moreover, the vertices that  are neither sources nor sinks and that are not
on the middle-line of $\Theta$ can have only degrees $(2,2)$ or $(1,1)$.
\end{itemize}
  The statement of the lemma clearly follows from the fact that no vertex of $\Theta$ has
  in- or out-degree greater than two.
\end{proof}
\begin{rem}\label{parity} Because of the alternating orientation of the edges, the following paths of the canonical path decomposition have an even number of edges:
\begin{itemize}
\item closed paths,
\item open paths beginning and ending at a source,
\item open paths beginning and ending at a sink,
\item open paths beginning and ending at the middle line.
\end{itemize}
\end{rem}

Now we define an {\em alternating coloring} of $\Theta$ as a
coloring of the edges of $\Theta$ in two colors in such a way that
the consecutive edges of each path of its canonical path decomposition are
colored differently (see Figure \ref{net:alternate}). By Remark \ref{parity}, such a coloring always exists.

\begin{figure}[h]
\begin{tikzpicture}
\draw[opacity=0.3,dotted] (3,-0.5) -- (3,2.5);
\draw [red,thick](0,0) -- (1,0);
\draw [green,thick](0,1) -- (1,0);
\draw [red,thick] (0,1) -- (1,1);
\draw [green,thick](0.3,2) -- (1,1);
\draw [red,thick](0.3,2) -- (3,2);
\draw [green,thick](2,1) -- (3,2);
\draw [red,thick](2,1) -- (3.5,1);
\draw [green,thick](2.5,0)-- (3.5,1) ;
\draw [red,thick](2.5,0) -- (3,0);
\draw [green,thick](5.5,1) -- (6,1);
\draw [red,thick](5.5,1) -- (6,0.5);
\draw [red,thick] (0,2) -- (0.3,2);
\draw [green,thick](0,2.5)  -- (0.3,2);
\draw [green,red](3,2) .. controls (4.5,3) .. (6,2);
\draw [green,thick](1,1) -- (2,1);
\draw [red,thick](1,1) .. controls (1.5,0.5) .. (2,1);
\draw [red,thick](3.5,1) -- (5.5,1);
\draw [green,thick](3.5,1) .. controls (4.5,2) .. (5.5,1);
\draw [green,thick](1,0) -- (2.5,0);
\draw [red,thick](1,0) .. controls (1.8,-0.6)  .. (2.5,0);
\end{tikzpicture}\caption{An alternating coloring of $\Theta$.} \label{net:alternate}
\end{figure}
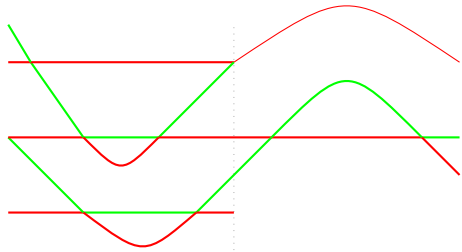

\begin{lem}  Assume that $\Theta$ has no vertices of degree $(1,1)$,
  and denote the elements of its path decomposition by $Q$.
  \begin{enumerate}
  \item Then there are precisely $2^{|Q|}$ alternating colorings of
    $\Theta$, corresponding to a coloring of each of the $|Q|$ paths
    chosen independently.
  \item The alternating colorings of $\Theta$ coincide with the valid
    colorings of $\Theta$.
  \end{enumerate}
\end{lem}
\begin{proof}
  Clearly, every path has precisely two alternating colorings, and the
  colorings of different paths are independent of each other. This
  implies the first statement.

  A coloring of the edges of $\Theta$ is valid if and
  only if
  \begin{itemize}
  \item at any vertex having in-degree 2, the two incoming edges are
    colored differently,
  \item at any vertex having out-degree 2, the two outgoing edges are
    colored differently. In particular, the two edges of a double edge
    have different colors.
  \end{itemize}
  These conditions coincide with the definition of an alternating
  coloring.
\end{proof}

Now we prove a  somewhat strengthened version of the decomposition
statement, which will imply our theorem:
\begin{prop}
  \label{hornprop}
Let $\Theta$ be a generalized planar network with the following
properties:
\begin{itemize}
\item$\Theta$ is  of type $[2k-1,2i-1]$,
\item the sources of $\Theta$ have degrees $(0,1)$ or $(0,2)$,
\item all vertices of $\Theta$, apart from the sources, have degrees
  $(d_1,d_2)$ with $2\ge d_1\ge d_2$.
\end{itemize}
Then there is $\tilde\alpha\in\ppt ki$ and $\tilde\beta\in\ppt{k-1}{i-1}$
such that $\tilde\alpha\cup\tilde\beta=\Theta$.
\end{prop}
\begin{proof}
  We can partition the set $Q$ of path components of $\Theta$ as
  follows:
\[  Q=Q_{00}\cup Q_{L0}\cup Q_{0R} \cup Q_{LR}\cup Q_{LL}\cup Q_{RR}\cup Q_{cl},
\]
where $Q_{cl}$ consists of all closed paths and the two indices of the rest of the $Q$'s indicate the beginning and the end of the path
with the convention that
\begin{itemize}
\item $L$ stands for a source of $\Theta$,
\item $R$ stands for a sink of $\Theta$,
\item $0$ stands for an internal vertex of $\Theta$.
\end{itemize}
We must show that among the $2^{|Q|}$ alternating colorings of the paths
in $Q$, there is at least one for which precisely $k$ edges emanating
from a source (source-edges), and precisely $i$ edges ending in a sink
(sink-edges) are red. It is clear that in order to specify the
coloring of a path, it is sufficient to color its source-edge or
sink-edge whenever the path has one of these.

We note that the contribution of each of $Q_{00}$, $Q_{LL}$, $Q_{RR}$, and $Q_{cl}$ to the source-degree and the sink-degree is even, so their coloring is not essential. The total contribution of $Q_{L0}$ and $Q_{LR}$ to the source-degree and the total contribution of $Q_{0R}$ and $Q_{LR}$ to the sink-degree are odd. Therefore the coloring algorithm is as follows:
\begin{itemize}
\item if $|Q_{LR}|$ is even, then $|Q_{0R}|$ is odd and $|Q_{L0}|$
  is odd. Then we color in red
  \begin{itemize}
  \item the sink-edges of half of the paths in
  $Q_{LR}$ and of $(|Q_{0R}|+1)/2$ paths in $Q_{0R}$,
\item the source-edges of  $(|Q_{L0}|+1)/2$ paths in  $Q_{L0}$.
  \end{itemize}
\item if $|Q_{LR}|$ is odd, then $|Q_{0R}|$ and $|Q_{L0}|$
  are even. Then we color in red
  \begin{itemize}
  \item the sink-edges of  $(|Q_{LR}|+1)/2$ of the paths in
  $Q_{LR}$ and  of half of the paths in $Q_{0R}$,
\item the source-edges of half of paths in  $Q_{L0}$.
  \end{itemize}
\end{itemize}
 This coloring algorithm ensures that the number of paths whose source-edges are colored in red is greater by $1$ than the number of paths with source-edges colored in green. The same is true for the sink-edges. Thus we indeed obtain an element of $P_i^k\Theta$ colored in red and an element of $P_{i-1}^{k-1}\Theta$ colored in green.
\end{proof}

As we explained above, Proposition \ref{hornprop} implies  Inequality \eqref{eqm}.

The other two inequalities from \eqref{eq:GD} are proved in a similarly.  Consider the second inequality. Let
$\alpha\in\pp{k-1}{i-1}$ and $\beta\in\pp{k}{i+1}$. Then $\Theta=\alpha\cup\beta$ is a
generalized planar network of type $[2k-1,2i]$. As before, we obtain a
path decomposition $Q=Q_{00}\cup Q_{L0}\cup Q_{0R} \cup Q_{LR}\cup
Q_{LL}\cup Q_{RR}\cup Q_{cl}$ of $\Theta$ with $2^{|Q|}$ alternating colorings. Among these colorings, we must  find one with precisely $k$
source-edges and precisely $i$ sink-edges colored in red. For this,
we use the following algorithm:

\begin{itemize}
\item if $|Q_{LR}|$ is even, then $|Q_{0R}|$ is even while $|Q_{L0}|$
  is odd. Then we color in red
  \begin{itemize}
  \item the sink-edges of half of the paths in
  $Q_{LR}$ and in $Q_{0R}$, 
\item the source-edges of  $(|Q_{L0}|+1)/2$ paths in  $Q_{L0}$.
  \end{itemize}
\item if $|Q_{LR}|$ is odd, then $|Q_{0R}|$ is odd and $|Q_{L0}|$
  is even. Then we color in red
  \begin{itemize}
  \item the sink-edges of  $(|Q_{LR}|+1)/2$ of the paths in
  $Q_{LR}$ and  of $(|Q_{0R}|-1)/2$ of the paths in $Q_{0R}$, 
\item the source-edges of half of the paths in  $Q_{L0}$.
  \end{itemize}
\end{itemize}
This coloring algorithm ensures that the number of paths with red souce-edges is greater by 1 than the number of paths with green source-edges, whereas the number of paths with red sink-edges is equal to the number of green sink-edges, and so we  indeed obtain an element of $P_i^k$ colored in red and an element of $P_{i}^{k-1}$ colored in green.

Finally, for the third inequality, we have a generalized planar
network $\Theta$ of type $[2k,2i-1]$. The coloring procedure is similar to the other two cases, and
will be omitted.

We have thus proved inequalities \eqref{eq:GD}, which completes the proof of the theorem.

\begin{rem}A similar proof can be given for Theorem \ref{thm:GZ1}, which is the analog of  Theorem \ref{thm:Horn1} in the Gelfand--Zeitlin case.
\end{rem}

\subsection{Proof of Theorem  \ref{thm:Horn2}}

The strategy of the proof of Theorem \ref{thm:Horn2} is analogous to
that of Theorem \ref{thm:GZ2}.
The steps of the proof are as follows:

1. We identify the special collection $B=\{\beta(k,i);\; 0\leq i \leq
  k\leq n\}$. Denote by $\tau(k)$ the $k$th
  horizontal line of $\Gamma_0\circ\Delta_0$. Then 
\[     \beta(k,i)=\alpha(n-i,k-i)\cup\bigcup_{j=n-i+1}^n\tau(j).
\]

\begin{figure}[h]
\begin{tikzpicture}
\draw (-1,0) -- (8.5,0);
\draw (-1,0.5) -- (8.5,0.5);
\draw (-1,1) -- (8.5,1);
\draw (-1,1.5) -- (8.5,1.5);
\draw (-1,2) -- (8.5,2);
\draw (-1,2.5) -- (8.5,2.5);
\draw (0.5,2) -- (1,1.5);
\draw (2,2) -- (2.5,1.5);
\draw (3.5,2) -- (4,1.5);
\draw (5,2) -- (5.5,1.5);
\draw (1.5,1.5) -- (2,1);
\draw (3,1.5) -- (3.5,1);
\draw (4.5,1.5) -- (5,1);
\draw (2.5,1) -- (3,0.5);
\draw (4,1) -- (4.5,0.5);
\draw (3.5,0.5) -- (4,0);
\draw (-0.5,2.5) -- (0,2);
\draw (1,2.5) -- (1.5,2);
\draw (2.5,2.5) -- (3,2);
\draw (4,2.5) -- (4.5,2);
\draw (5.5,2.5) -- (6,2);
\draw[opacity=0.3] (6.5,0) -- (6.5,2.5);
\draw[ultra thick,red](-1,2.52) -- (8.5,2.52);
\draw[ultra thick,red](-1,2.02) -- (8.5,2.02);
\draw[ultra thick,red](-1,1.52) -- (3.02,1.52) -- (3.52,1.02) -- (4.02,1.02) -- (4.52,0.52) -- (6.5,0.52);
\draw[ultra thick,red](-1,1.02) -- (2.52,1.02) -- (3.02,0.52) -- (3.52,0.52) -- (4.02,0.02) -- (6.5,0.02);
\fill[black] (-1,0) circle (0pt) node[left]{1};
\fill[black] (-1,2.5) circle (0pt) node[left]{6};
\fill[black] (8.5,0) circle (0pt) node[right]{1};
\fill[black] (8.5,2.5) circle (0pt) node[right]{6};
\draw (2.5,3) node{$\Gamma_0$};
\draw (7.5,3) node{$\Delta_0$};
\draw (6,2.4) node[above,font=\small]{$\tau(6)$};
\draw (7,1.9) node[above,font=\small]{$\tau(5)$};
\end{tikzpicture}
 \caption{$\beta(4,2)$ for $n=6$, shown in solid red.}\label{2net:betas}
\end{figure}
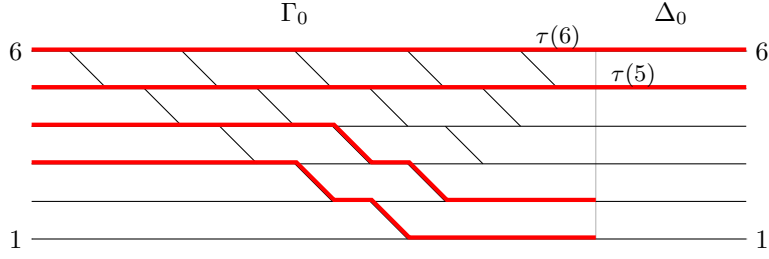

The following counterpart of Lemma \ref{tautology} is valid: the theorem follows from the
inclusion
\begin{equation}
  \label{inclusion}
  w_B^{-1}(\bar\C_3)\cap\{t_0^0=0\} \subset
  \{\epsilon\in\T^{E(\Gamma_0\circ\Delta_0)}|\;
  m^k_i\Gamma_0\Delta_0(\epsilon)=
  w_{\beta(k,i)}(\epsilon),\,0\leq i \leq
  k\leq n\};
\end{equation}
in words, we must show that if, for some weighting, the weights of
our special collection satisfy the inequalities defining $\C_3$ (where $t_0^0=0$), then the multipaths of our collection are maximal
for this weighting.

2. We identify the regions corresponding to the inequalities \eqref{eq:GD} for $m^k_i=w_{\beta(k,i)}$. We denote the cells of $\Gamma_0$ as before and the cells of $\Delta_0$ by $[i,i+1]$ for $i=0,\dots,n$.  The first inequality corresponds to $r^\searrow_{[k,i]}$, the third  to $r^\nearrow_{[k,i]}$, while the second inequality corresponds to
\[   r^\rightarrow_{[k,i]}=c_{[k,i]}+c_{[k,i+1]}+\dots+c_{[k,k]}+c_{[k,k+1]}
\]
(see Figure \ref{2net:regions-right}).

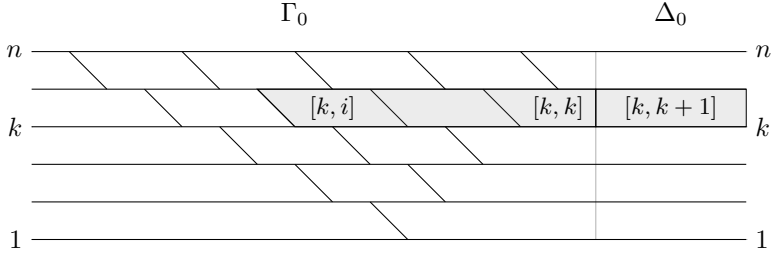
\begin{figure}[h]
\begin{tikzpicture}
\draw (-1,0) -- (8.5,0);
\draw (-1,0.5) -- (8.5,0.5);
\draw (-1,1) -- (8.5,1);
\draw (-1,1.5) -- (8.5,1.5);
\draw (-1,2) -- (8.5,2);
\draw (-1,2.5) -- (8.5,2.5);
\draw (0.5,2) -- (1,1.5);
\draw (2,2) -- (2.5,1.5);
\draw (3.5,2) -- (4,1.5);
\draw (5,2) -- (5.5,1.5);
\draw (1.5,1.5) -- (2,1);
\draw (3,1.5) -- (3.5,1);
\draw (4.5,1.5) -- (5,1);
\draw (2.5,1) -- (3,0.5);
\draw (4,1) -- (4.5,0.5);
\draw (3.5,0.5) -- (4,0);
\draw (-0.5,2.5) -- (0,2);
\draw (1,2.5) -- (1.5,2);
\draw (2.5,2.5) -- (3,2);
\draw (4,2.5) -- (4.5,2);
\draw (5.5,2.5) -- (6,2);
\draw[opacity=0.3] (6.5,0) -- (6.5,2.5);
\draw[fill=gray,fill opacity=0.15, line width=0.001] (2.5,1.5) -- (6.5,1.5) -- (6.5,2) -- (2,2) -- cycle;
\draw[fill=gray,fill opacity=0.15, line width=0.001] (6.5,1.5) -- (6.5,2) -- (8.5,2) -- (8.5,1.5) -- cycle;
\fill[black] (-1,0) circle (0pt) node[left]{1};
\fill[black] (-1,1.5) circle (0pt) node[left]{$k$};
\fill[black] (-1,2.5) circle (0pt) node[left]{$n$};
\fill[black] (8.5,0) circle (0pt) node[right]{1};
\fill[black] (8.5,1.5) circle (0pt) node[right]{$k$};
\fill[black] (8.5,2.5) circle (0pt) node[right]{$n$};
\path (3,1.75) node [font=\small] {$[k,i]$};
\path (6,1.75) node [font=\small] {$[k,k]$};
\path (7.5,1.75) node [font=\small] {$[k,k+1]$};
\draw (2.5,3) node{$\Gamma_0$};
\draw (7.5,3) node{$\Delta_0$};
\end{tikzpicture}
 \caption{$r^\rightarrow_{[k,i]}$.}\label{2net:regions-right}
\end{figure}

We thus obtain the graphical representation:
\begin{equation}
  \label{graphrel}
   w_B(\epsilon)\in\C_3\Leftrightarrow
r^\nearrow_{[k,i]}(\epsilon)\geq0,\,r^\searrow_{[k,i]}(\epsilon)\leq0,\,r^\rightarrow_{[k,i]}(\epsilon)\geq0.
\end{equation}

3. Now we must show that if condition \eqref{graphrel} holds for a
weighting $\epsilon$ of $\Gamma_0\circ\Delta_0$, i.e., if
$r^\nearrow_{[k,i]},-r^\searrow_{[k,i]},r^\rightarrow_{[k,i]}\geq0$
for all $i\leq k< n$, then the collection $B$ is maximal for
$\epsilon$.

From now on, we assume that we have a fixed weighting $\epsilon$ of
$\Gamma_0\circ\Delta_0$ satisfying \eqref{graphrel}, and we will often drop
$\epsilon$ from the notation. When we say that  $\alpha\in\ppo
ki$ is {\em maximal}, we  mean that
$w_\alpha(\epsilon)=m^k_i\Gamma_0\Delta_0(\epsilon)$.

Note that every
$\alpha\in\pp ki$ has two decompositions:
\[    \alpha=\gamma\cup\delta=\alpha'\cup\alpha'',\,
\gamma\in P_k\Gamma_0,\,\delta\in P_i\Delta_0,\,\alpha'\in
P_i(\Gamma_0\circ\Delta_0),\,\alpha''\in P_{k-i}\Gamma_0.
\]

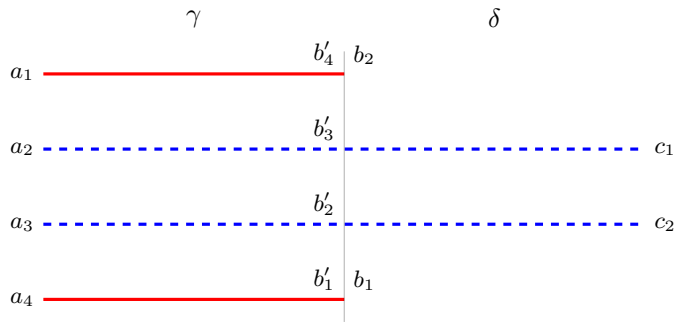
\begin{figure}[h]
\begin{tikzpicture}
\draw[very thick,red] (0,0) -- (4,0);
\draw[very thick,blue,dashed] (0,1) -- (8,1);
\draw[very thick,blue,dashed] (0,2) -- (8,2);
\draw[very thick,red] (0,3) -- (4,3);
\draw[opacity=0.3] (4,-0.3) -- (4,3.3);
\draw(0,3) node[left,font=\small]{$a_1$};
\draw(0,2) node[left,font=\small]{$a_2$};
\draw(0,1) node[left,font=\small]{$a_3$};
\draw(0,0) node[left,font=\small]{$a_4$};
\draw(4,0) node[above right,font=\small]{$b_1$};
\draw(4,3) node[above right,font=\small]{$b_2$};
\draw(4,0) node[above left,font=\small]{$b'_1$};
\draw(4,1) node[above left,font=\small]{$b'_2$};
\draw(4,2) node[above left,font=\small]{$b'_3$};
\draw(4,3) node[above left,font=\small]{$b'_4$};
\draw(8,2) node[right,font=\small]{$c_1$};
\draw(8,1) node[right,font=\small]{$c_2$};
\draw (2,3.5) node[above]{$\gamma$};
\draw (6,3.5) node[above]{$\delta$};
\end{tikzpicture}
 \caption{Path $\alpha'$ in dashed blue, path $\alpha''$ in solid red.}\label{2net:sequences}
\end{figure}

We record the endpoints of these multipaths as follows (see Figure \ref{2net:sequences}):
\begin{itemize}
\item denote by   $\ba=(a_1,\dots,a_k)$ the decreasing sequence of
  sources of $\gamma$;
\item denote by $\bb=(b_1,\dots,b_{k-i})$ the increasing sequence of
  sinks of $\alpha''$,
\item denote by $\bb'=(b_1',\dots,b_{k}')$ the increasing sequence of
  sinks of $\gamma$. Note that $\bb$ is a subsequence of $\bb'$.
\item denote by
  $\bc=(c_1,\dots,c_i)$ the decreasing sequence of sinks of $\alpha'$
  (or $\delta$).
\end{itemize}
When necessary, we will indicate the index $\alpha$ explicitly
by writing, for example, $\bb_\alpha$ instead of $\bb$.

Now, similarly to Lemmas \ref{sinks} and \ref{sources}, we can
formulate a sequence of lemmas normalizing the form of a maximal path
that eventually lead to the statement that $\beta(k,i)$ is maximal for
the weighting $\epsilon$.

\begin{lem}\label{hornb}
  Let $\alpha\in \ppo ki$ be a maximal path for which $\sum \bb$
  minimal. Then $b_{k-i}=b'_{k-i}$, which means that $\alpha''$
  consists of the lowest $k-i$ paths of $\gamma$.
\end{lem}

\begin{proof}
To prove this, we observe that
\begin{equation}
  \label{cineq}
 c_{[j,j+1]}(\epsilon)\geq0\text{ for }1\leq j\leq n-1,
\end{equation}
since $c_{[j,j+1]}=r^\rightarrow_{[j,j]}-r^\searrow_{[j,j]}.$
Assume, contrary to the statement of the lemma, that for some $j\leq
k-i$, we have $b_j>b_j'$  (see  Figure \ref{2net:upperpaths}).
Then we obtain
$$w_{\tilde\alpha}(\epsilon)=w_{\alpha}(\epsilon)+c_{[b_j',b_j'+1]}(\epsilon)+\dots+c_{[b_j-1,b_j]}(\epsilon)\geq
w_\alpha(\epsilon),$$ for a multipath $\tilde\alpha$. Clearly,
$\tilde\alpha$ is maximal with
$\sum\bb_{\tilde\alpha}<\sum\bb_\alpha$, which contradicts our
assumption on $\alpha$.
\end{proof}

\begin{figure}[h]
\begin{tikzpicture}
\draw (-1,0) -- (8.5,0);
\draw (-1,0.5) -- (8.5,0.5);
\draw (-1,1) -- (8.5,1);
\draw (-1,1.5) -- (8.5,1.5);
\draw (-1,2) -- (8.5,2);
\draw (-1,2.5) -- (8.5,2.5);
\draw (0.5,2) -- (1,1.5);
\draw (2,2) -- (2.5,1.5);
\draw (3.5,2) -- (4,1.5);
\draw (5,2) -- (5.5,1.5);
\draw (1.5,1.5) -- (2,1);
\draw (3,1.5) -- (3.5,1);
\draw (4.5,1.5) -- (5,1);
\draw (2.5,1) -- (3,0.5);
\draw (4,1) -- (4.5,0.5);
\draw (3.5,0.5) -- (4,0);
\draw (-0.5,2.5) -- (0,2);
\draw (1,2.5) -- (1.5,2);
\draw (2.5,2.5) -- (3,2);
\draw (4,2.5) -- (4.5,2);
\draw (5.5,2.5) -- (6,2);
\draw[opacity=0.3] (6.5,0) -- (6.5,2.5);
\draw[fill=gray,fill opacity=0.15, line width=0.001] (6.5,0) -- (6.5,1) -- (8.5,1) -- (8.5,0) -- cycle;
\draw[ultra thick,red] (-1,2.02) -- (0.53,2.02) -- (1.03,1.52) -- (8.5,1.52);
\draw[dashed,ultra thick,blue] (-1,1.97) -- (0.5,1.97) -- (1,1.47) -- (8.5,1.47);
\draw[ultra thick,red] (-1,1.02) -- (6.5,1.02);
\draw[dashed,ultra thick,blue] (-1,0.97) -- (8.5,0.97);
\draw[ultra thick,red] (-1,0.52) -- (3.53,0.52) -- (4.03,0.02) -- (8.5,0.02);
\draw[dashed,ultra thick,blue] (-1,0.47) -- (3.5,0.47) -- (4,-0.03) -- (6.5,-0.03);
\draw (2.5,3) node{$\Gamma_0$};
\draw (7.5,3) node{$\Delta_0$};
\end{tikzpicture}
 \caption{$w_\alpha$ (solid) $=$ $w_{\tilde\alpha}$ (dashed) $-$ $\sum c_{[l,l+1]}$ (shaded).}\label{2net:upperpaths}
\end{figure}
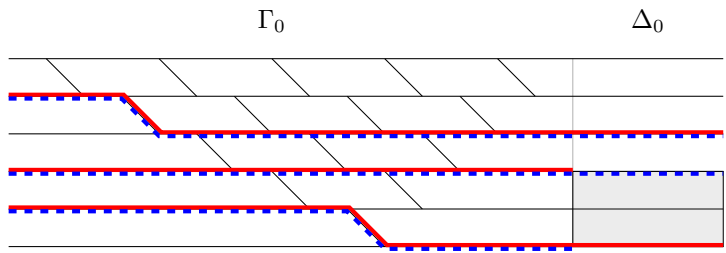

\begin{lem}\label{horna1}
   Let $\alpha\in \ppo ki$ be a maximal path for which
   $b_{k-i}=b'_{k-i}$ and the value of $a_1$ is maximal. Then $a_1=n$.
\end{lem}
The proof of this statement is identical to the first part of the
proof of Lemma \ref{sources}  (see Figure \ref{2net:uppersources}).

\begin{figure}[h]
\begin{tikzpicture}
\draw (-1,0) -- (8.5,0);
\draw (-1,0.5) -- (8.5,0.5);
\draw (-1,1) -- (8.5,1);
\draw (-1,1.5) -- (8.5,1.5);
\draw (-1,2) -- (8.5,2);
\draw (-1,2.5) -- (8.5,2.5);
\draw (0.5,2) -- (1,1.5);
\draw (2,2) -- (2.5,1.5);
\draw (3.5,2) -- (4,1.5);
\draw (5,2) -- (5.5,1.5);
\draw (1.5,1.5) -- (2,1);
\draw (3,1.5) -- (3.5,1);
\draw (4.5,1.5) -- (5,1);
\draw (2.5,1) -- (3,0.5);
\draw (4,1) -- (4.5,0.5);
\draw (3.5,0.5) -- (4,0);
\draw (-0.5,2.5) -- (0,2);
\draw (1,2.5) -- (1.5,2);
\draw (2.5,2.5) -- (3,2);
\draw (4,2.5) -- (4.5,2);
\draw (5.5,2.5) -- (6,2);
\draw[opacity=0.3] (6.5,0) -- (6.5,2.5);
\draw[fill=gray,fill opacity=0.15, line width=0.001] (-1,2) -- (0,2) -- (-0.5,2.5) -- (-1,2.5) -- cycle;
\draw[dashed,ultra thick,blue] (-1,2.52) -- (-0.47,2.52)  -- (0.03,2.02)-- (0.53,2.02) -- (1.03,1.52) -- (8.5,1.52);
\draw[ultra thick,red] (-1,1.97) -- (0.5,1.97) -- (1,1.47) -- (8.5,1.47);
\draw[dashed,ultra thick,blue] (-1,1.02) -- (8.5,1.02);
\draw[ultra thick,red] (-1,0.97) -- (8.5,0.97);
\draw[dashed,ultra thick,blue] (-1,0.52) -- (3.53,0.52) -- (4.03,0.02) -- (6.5,0.02);
\draw[ultra thick,red] (-1,0.47) -- (3.5,0.47) -- (4,-0.03) -- (6.5,-0.03);
\draw (2.5,3) node{$\Gamma_0$};
\draw (7.5,3) node{$\Delta_0$};
\end{tikzpicture}
 \caption{Raising the top source by adding $ r^\nearrow_{[j,0]}$ (shaded).}\label{2net:uppersources}
\end{figure}

\begin{lem}\label{hornc1}
  Let $i>0$ and let $\alpha\in \ppo ki$ be a maximal multipath for which
  $b_{k-i}=b'_{k-i}$, $a_1=n$ and the value of $c_1$ is maximal. Then
  $c_1=n$, and hence $\tau(n)\subset\alpha$.
\end{lem}
\begin{proof}
  Since $b_{k-i}=b'_{k-i}$ for $i>0$, we know that the highest path
  of $\alpha$ is a path in $\Gamma_0\circ\Delta_0$. Assume that $c_1\neq
  n$. Then, as shown in
  Figure \ref{2net:uppersinks}, for an appropriate sequence $q_j,\,
  j=c_1,\dots,n-1$, we have
\[    w_{\tilde\alpha}=w_\alpha +\sum_{j=c_1}^{n-1} r^\rightarrow_{[j,q_j]}\geq0,
\]
where $\tilde\alpha\in\ppo ki$. Again, this contradicts the
assumptions of the lemma, hence $c_1=n$.
\end{proof}

\begin{figure}[h]
\begin{tikzpicture}
\draw (-1,0) -- (8.5,0);
\draw (-1,0.5) -- (8.5,0.5);
\draw (-1,1) -- (8.5,1);
\draw (-1,1.5) -- (8.5,1.5);
\draw (-1,2) -- (8.5,2);
\draw (-1,2.5) -- (8.5,2.5);
\draw (0.5,2) -- (1,1.5);
\draw (2,2) -- (2.5,1.5);
\draw (3.5,2) -- (4,1.5);
\draw (5,2) -- (5.5,1.5);
\draw (1.5,1.5) -- (2,1);
\draw (3,1.5) -- (3.5,1);
\draw (4.5,1.5) -- (5,1);
\draw (2.5,1) -- (3,0.5);
\draw (4,1) -- (4.5,0.5);
\draw (3.5,0.5) -- (4,0);
\draw (-0.5,2.5) -- (0,2);
\draw (1,2.5) -- (1.5,2);
\draw (2.5,2.5) -- (3,2);
\draw (4,2.5) -- (4.5,2);
\draw (5.5,2.5) -- (6,2);
\draw[opacity=0.3] (6.5,0) -- (6.5,2.5);
\draw[fill=gray,fill opacity=0.15, line width=0.001]  (0,2) -- (0.5,2) -- (1,1.5) -- (8.5,1.5) -- (8.5,2.5) -- (-0.5,2.5) -- cycle;
\draw[ultra thick,red] (-1,2.52) -- (-0.47,2.52)  -- (0.03,2.02)-- (0.53,2.02) -- (1.03,1.52) -- (8.5,1.52);
\draw[dashed,ultra thick,blue] (-1,2.47) --  (8.5,2.47);
\draw[ultra thick,red] (-1,1.02) -- (8.5,1.02);
\draw[dashed,ultra thick,blue] (-1,0.97) -- (8.5,0.97);
\draw[ultra thick,red] (-1,0.52) -- (3.53,0.52) -- (4.03,0.02) -- (6.5,0.02);
\draw[dashed,ultra thick,blue] (-1,0.47) -- (3.5,0.47) -- (4,-0.03) -- (6.5,-0.03);
\draw (2.5,3) node{$\Gamma_0$};
\draw (7.5,3) node{$\Delta_0$};
\end{tikzpicture}
 \caption{Raising the top sink by adding $\sum r^\rightarrow_{[j,l]}$ (shaded).}\label{2net:uppersinks}
\end{figure}

Now, using the proof of Theorem \ref{thm:GZ2}, we can quickly finish the proof of Theorem \ref{thm:Horn2} by
induction on $n$.

Indeed, Theorem \ref{thm:Horn2} is trivial for $n=1$.
Assume that the theorem has already been proved for the case of  rank $n-1$, and that
$\epsilon$ is a weighting satisfying \eqref{graphrel}. Then the fact
that $\beta(k,0)$ is maximal for $k=1,\dots,n$ is a special case of
Theorem \ref{thm:GZ2}. If $k\geq i>0$, then the sequence of lemmas
above shows that there is at least one maximal $\alpha\in\ppo ki$
containing $\tau(n)$. Proving that $\beta(k,i)$ is maximal
is then equivalent to showing that $\beta(k-1,i-1)$ is maximal for the
restriction of $\epsilon$ to the planar network obtained by removing
$\tau(n)$ from $\Gamma_0\circ\Delta_0$. This, however, follows from
our inductive assumption, which completes the proof.

\end{document}